\documentclass[10pt]{article}
\usepackage{amsmath,amssymb,amsthm, epsfig}

\numberwithin{equation}{section}

\usepackage{mathtools}
\mathtoolsset{showonlyrefs}
\usepackage[colorlinks=true,linkcolor=blue,urlcolor=blue,citecolor=green!40!black]{hyperref}
\usepackage{amsmath}

\usepackage{amssymb}

\usepackage{xcolor}

\usepackage{ulem}

\usepackage{epsfig}

\usepackage[mathscr]{eucal}

\usepackage{stackrel}

\usepackage[nottoc,notlot,notlof]{tocbibind}

\usepackage{tikz}

\usepackage{multicol}

\usepackage{mathptmx} % rm & math
\title{\bf An optimization problem with volume constraint with applications to optimal mass transport}
\author{
Jo\~{a}o Vitor da Silva\thanks{\noindent CONICET and Universidad de Buenos Aires. FCEyN, Department of Mathematics. Ciudad Universitaria-Pabell\'{o}n I-(C1428EGA), \href{tel:54 11 5285 7618/19}{+54 11 5285 7618/19} - Buenos Aires, Argentina. \noindent \texttt{E-mail address: \url{jdasilva@dm.uba.ar} and \url{jrossi@dm.uba.ar}}},\,\,\,
        Leandro M. Del Pezzo\thanks{CONICET and UTDT,
		Departamento de Matem\'aticas y Estad\'istica, Universidad Torcuato Di Tella,
		Av. Figueroa Alcorta 7350 (C1428BCW), \href{tel:54 11 5169 7000}{+54 11 5169 7000},  Buenos Aires, Argentina.
		 \noindent \texttt{E-mail address: \url{ldelpezzo@utdt.edu}}}\,\,\,
		 and \,\,\,Julio D. Rossi\footnotemark[1]
}

\newlength{\hchng}
\newlength{\vchng}
\def \R {\mathbb{R}}

\def \div {\mathrm{div}}

\def \dist {\mathrm{dist}}

\def \suchthat {\colon}

\def \H {\mathcal{H}^{N-1}}

\def \Leb {\mathscr{L}^N}
\def \tr {\mathrm{Tr}}

\newcommand{\defeq}{\mathrel{\mathop:}=}

\newtheorem{theorem}{Theorem}[section]

\newtheorem{lemma}[theorem]{Lemma}

\newtheorem{corollary}[theorem]{Corollary}

\theoremstyle{definition}

\newtheorem{definition}[theorem]{Definition}

\newtheorem{example}[theorem]{Example}

\theoremstyle{remark}

\newtheorem{remark}[theorem]{Remark}

\numberwithin{equation}{section}

\newcommand{\intav}[1]{\mathchoice {\mathop{\vrule width 6pt height 3 pt depth  -2.5pt
\kern -8pt \intop}\nolimits_{\kern -6pt#1}} {\mathop{\vrule width
5pt height 3  pt depth -2.6pt \kern -6pt \intop}\nolimits_{#1}}
{\mathop{\vrule width 5pt height 3 pt depth -2.6pt \kern -6pt
\intop}\nolimits_{#1}} {\mathop{\vrule width 5pt height 3 pt depth
-2.6pt \kern -6pt \intop}\nolimits_{#1}}}

\begin{document}
\maketitle

\begin{abstract}
    In this manuscript we study the following optimization problem with volume constraint:
    $$
        \text{min} \left\{\frac{1}{p}\int_{\Omega} |\nabla v|^pdx-
        \int_{\partial \Omega} gv\,d\H \colon v \in W^{1, p}
        \left(\Omega\right), \text{ and }
        \Leb(\{v>0\}) \leq \alpha\right\}.
    $$
    Here $\Omega \subset \R^N$ is a bounded and smooth domain, $g$ is a
    continuous function and $\alpha$ is a fixed constant such that
    $0< \alpha< \Leb(\Omega)$.
    Under the assumption that
    $\displaystyle \int_{\partial \Omega} g(x)d\H >0$
    we prove that a minimizer exists and satisfies
    $$
        \left\{
            \begin{array}{ccl}
                -\Delta_p u_p  =  0 &\text{in }
                &\{u_p>0\} \cup \{u_p<0\}, \\[5pt]
                |\nabla u_p|^{p-2}\frac{\partial u_p}{\partial \nu}  =  g
                 & \text{on }  &\partial \Omega
                 \cap\partial(\{u_p>0\} \cup \{u_p<0\} ) ,\\[5pt]
                \Leb(\{u_p>0\}) = \alpha. & &
            \end{array}
        \right.
    $$

    Next, we analyze the limit as $p\to \infty$. We obtain that any sequence of weak solutions converges, up to a
    subsequence, $\displaystyle \lim_{p_j \to \infty} u_{p_j}(x)=u_{\infty}(x)$,
    uniformly in $\overline{\Omega}$, and uniform limits, $u_\infty$, are
    solutions to the maximization problem with volume constraint
    $$
        \max\left\{ \int_{\partial \Omega} gv\,d\H \suchthat v \in W^{1, \infty}
        \left(\Omega\right), \|\nabla v\|_{L^{\infty}(\Omega)}\leq 1
        \text{ and }\Leb(\{v>0\}) \leq \alpha\right\}.
    $$
    Furthermore, we obtain the limit equation that is verified by $u_\infty$ in the viscosity sense. Finally, it turns out that such a limit variational problem is connected to the Monge-Kantorovich mass transfer problem with the involved measures are supported on $\partial \Omega$ and along the limiting free boundary, $\partial \{u_{\infty} \neq 0\}$. Furthermore, we show some explicit examples of solutions for certain configurations of the domain and data.
\end{abstract}
%\tableofcontents

\newpage
\section{Introduction}

\subsection*{Motivation and historic overview}

    In shape optimization theory an \textit{Optimal Design Problem} under a volume constraint reads as follows: For an $\Omega \subset \R^N$ (smooth and bounded domain) and $0< \alpha < \Leb(\Omega)$ a fixed amount, we would like to find a best configuration $\mathcal{O} \subset \Omega$ such that minimizes a functional (cost) associated to a certain process, under the prescription of the maximum volume to be used. Mathematically this can be written as
\begin{equation}\label{ODP}
    \text{min} \,\,\left\{\mathfrak{J}_{\alpha}[u_{\Xi}]\suchthat \,\, u_{\Xi} \in \mathbb{X}(\Omega, \R)\,(\text{admissible class}), \,\,\Xi \subset \Omega \,\, \text{such that} \,\,u_{\Xi} >0  \,\,\, \text{in}\,\,\,\Xi \,\,\,\text{and}\,\,\, 0<\Leb(\Xi) \leq \alpha\right\}.
\end{equation}
    In several situations the functional $\mathfrak{J}_{\alpha}[u_{\Xi}]$ admits a variational representation, whose involved extremal functions are linked to the competing configuration $\Xi$ via a prescribed PDE. Some examples of such models appear as elliptic PDEs (eigenvalue problems with geometric constraints, shape optimization problems with constrained perimeter or volume), optimal design of semiconductor devices and problems in structural optimization, optimization problems with free boundaries, just to mention a few (cf. \cite{BucBut} for a large number of illustrative examples).

    Concerning free boundary optimization problems under volume constraint, its beginning dates back to the middle 80s. In the seminal work \cite{AAC} the authors study existence, regularity and geometric properties for minimizers of the optimization problem
    $$
    \displaystyle \min \,\,\left\{ \int_{\Omega} |\nabla v|^2dx: v \in W^{1.
    , p} (\Omega), \,\,\Delta v = 0 \,\,\, \text{in} \,\,\, \{v>0\}\cap \Omega, \,\,\,  u=g \,\,\, \text{on}\,\,\, \partial \Omega \,\,\, \mbox{and} \,\,\, \Leb(\{v=0\}) = \alpha\right\}.
    $$
    In the same direction, we also quote \cite{BMW} and \cite{OliTei}, where optimal design problems governed by quasi-linear operators of $p$-Laplace type were studied (the associated functional is $\displaystyle \mathfrak{J}_{\alpha}[v_{\Xi}] = \int_{\Omega} |\nabla v_{\Xi}|^pdx$). See also \cite{Tei1,Tei2,Tei3} and references therein concerning shape optimization problems in heat conduction, in this case $u$ represents the temperature in $\Omega$ of a heated body with non-constant prescribed temperature distribution $g$ on the boundary.

    We finish this quick overview by commenting the limiting (as $p\to \infty$) optimization problem treated in \cite{RT} (cf. \cite{RW} for a corresponding problem in the two-phase scenery and \cite{daSR} for a nonlocal counterpart). There it is considered the following limiting problem:

\begin{equation}\label{LimProb}
\displaystyle \min\left\{\sup_{x, y \in \Omega \atop{x\neq y}} \frac{|v(x)-v(y)|}{|x-y|}: v\in W^{1, \infty}(\Omega), \,\,\, v=g\geq 0 \,\, \text{on}\,\, \partial \Omega\,\,\,\text{and}\,\,\,\Leb(\{v>0\}) \leq \alpha\right\}.
\end{equation}
    In particular, in \cite{RT} extremals for \eqref{LimProb} are obtained as limit points of minimizers $(u_p)_{p \geq 2}$ of the following free boundary optimization problem:
\begin{equation}\label{p-LimProb}
\displaystyle \min\left\{\int_{\Omega} |\nabla u_p|^p: u_p\in W^{1, p}(\Omega), \,\,\, \Delta_p\, u_p = 0 \,\,\,\text{in}\,\,\,
\{u_p>0\}, \,\,\,u_p=g \geq 0 \,\, \text{on}\,\, \partial \Omega\,\,\,\text{and}\,\,\,\Leb(\{u>0\}) \leq \alpha\right\}.
\end{equation}
    Furthermore, such limit solutions verify
    $$
  \left\{
    \begin{array}{rclcl}
   \Delta_{\infty} u_{\infty}(x) & = & 0 & \mbox{in} & \{u_{\infty}>0\}, \\
   u_{\infty}(x) & = & g(x) & \mbox{on} & \partial \Omega,
   \end{array}
    \right.
    $$
    in the viscosity sense (Section \ref{Prelim} for such a concept), where
    $$
  \Delta_{\infty} \, v(x) \defeq  \nabla v^T(x)D^2v(x)\nabla v(x) = \sum_{i, j=1}^{N} \frac{\partial v}{\partial x_j}(x)\frac{\partial^2 v}{\partial x_j \partial x_i}(x) \frac{\partial v}{\partial x_i}(x)
    $$
    is the nowadays well-known \textit{$\infty$-Laplace operator}, which is naturally associated to \textit{Absolutely Minimizing Lipschitz Extensions} (cf. \cite{Arons} and \cite{ACJ} for comprehensive surveys about this subject). Notice that,  the $\infty-$Laplacian is a degenerate elliptic operator with non-divergence structure, see Section \ref{Prelim} for more details.

    With regards to nonlinear PDEs with Neumann type boundary conditions and viscosity solutions involving the outer normal derivative, i.e., $\frac{\partial u}{\partial \eta}$, the corresponding theory is quite more recent and we must quote \cite{Barl,CIL,Cran,Ish} and \cite{IL} as precursor works. In particular, such references establish uniqueness, comparison theorems, H\"{o}lder and Lipschitz regularity for solutions of general fully nonlinear elliptic equations (under suitable structural assumptions).

    On the other hand, in \cite{G-AMPR} it is studied the Neumann problem for the $\infty-$Laplace operator. The approach used there consists of analysing the limit as $p\to \infty$ of solutions to
    $$
  \left\{
    \begin{array}{rclcl}
   -\Delta_p u_p(x) & = & 0 & \mbox{in} &  \Omega,  \\
   |\nabla u_p(x)|^{p-2}\frac{\partial u_p}{\partial \eta}(x) & = & g(x) & \mbox{on} & \partial \Omega,
   \end{array}
    \right.
    $$
    with a continuous boundary flow $g$ verifying $\displaystyle \int_{\partial \Omega} g= 0$. In particular, it is proved that there exist limit points of $(u_p)_{p \geq 2}$ as $p \to \infty$. Furthermore, such limit points are maximizers of following variational problem:
\begin{equation}\label{EqLMP}
  \max\left\{ \int_{\partial \Omega} gv\,d\H \suchthat v \in W^{1, \infty}\left(\Omega\right),\,\,  \|\nabla v\|_{L^{\infty}(\Omega)}\leq 1 \,\,\text{and} \,\, \int_{\Omega} v = 0\right\}.
\end{equation}
    Another important piece of information is that limit points are viscosity solutions to
    $-\Delta_{\infty} u_{\infty}(x)  =  0 $ in $\Omega$ with $\mathrm{H}(x, u, \nabla u)  =  0$ on $ \partial \Omega$, a boundary condition that depends only on the sign of $g$, see \cite[Theorem 1.2]{G-AMPR} for more details.

\subsection*{Statement of the main results}

    Our main goal is the study of quasi-linear operators with $p$-Laplacian type structure with a volume constraint and Neumann boundary conditions and pass to the limit as $p\to \infty$. We consider the following optimization problem:
    {\small{
\begin{equation}\tag{{\bf $\mathfrak{P}_p$}}\label{eqMin_ps}
   \displaystyle  \mathfrak{P}_{p}[\alpha] \coloneqq
    \min \left\{\frac{1}{p}\int_{\Omega} |\nabla v|^pdx- \int_{\partial \Omega} gv\,d\H \suchthat v \in W^{1, p}\left(\Omega\right) \,\,\, \text{and} \,\,\, \Leb(\{v>0\}) \leq \alpha\right\}.
\end{equation}}}

    This kind of model (involving the $p-$Laplacian operator with Neumann boundary conditions) appears in a number of structural optimization, shape optimization and optimal design problems in pure and applied mathematics, as well as in the theory of some non-Newtonian fluids, reaction diffusion problems, etc. From an applied point of view one can think that we are prescribing the flux (a balance) across the boundary and trying to find the best of all configurations which minimizes a certain (physical) cost within a prescribed objective (class of admissible profiles) and a given set of geometrical limitations (constrained volume) in our procedure (cf. \cite{BucBut,Diaz} and references therein for nice essays about shape optimization and nonlinear PDEs theory, and compare with \cite{AAC,daSR,BMW,OliTei, Tei1, Tei2} and \cite{Tei3} for optimal design problems with constrained volume and Dirichlet boundary condition).

    For a datum $g$ such that $\displaystyle \int_{\partial \Omega} g(x)d\H >0$ the minimization problem admits at least one solution, but its existence is a non-trivial task, see Remark \ref{rem24} for more details. In this case, existence of a minimizer follows by using the direct method in calculus of variations, key tools comes from mathematical analysis and the construction of a suitable competitor profile in \eqref{eqMin_ps}.

 \begin{theorem}[{Existence of minimizers}]\label{ThmExistMin.intro}
        Let $p>N,$ $g\in L^1(\partial\Omega)$ be such that
        \[
            \int_{\partial \Omega} g \,d\H >0
        \]
        and $0< \alpha < \Leb(\Omega)$ fixed.
        Then, there is at least one function $u_p$
        solving \eqref{eqMin_ps}.

    Moreover, any minimizer $u_p$ is a weak solution to the following Neumann problem:
\begin{equation}\label{eqp}
  \left\{
\begin{array}{rclcl}
   -\Delta_p \,u_p(x) & = & 0 & \mbox{in} & \{u_p> 0\} \cup \{u_p< 0\}, \\
   |\nabla u_p(x)|^{p-2}\frac{\partial u_p}{\partial \eta}(x) & = & g(x) & \mbox{on} & \partial \Omega
   \cap \partial(\{u_p> 0\} \cup \{u_p< 0\}),
   \end{array}
\right.
\end{equation}
    and verifies
        \[
            \Leb(\{u>0\})=\alpha.
        \]

        In addition, if the domain is a ball, $\Omega=B_1(0)$ and $g$ is
        non-negative, spherically symmetric and strictly spherically decreasing
        with respect to an axis,
        then every minimizer is also spherically symmetric on $\partial B_1(0)$  with respect to this axis.
    \end{theorem}

    Notice that we don't have $|\nabla u_p(x)|^{p-2}\frac{\partial u_p}{\partial \eta}(x) = g(x)$ on the whole
    $\partial \Omega$. In fact, it could happen that the solution vanish on some part of $\partial \Omega$ and
    the Neumann boundary condition does not hold there, see Remark \ref{rem.33} for a simple one-dimensional
    example where this phenomenon takes place.

    It is worth to highlight that analytical and geometric features of the limiting (as $p\to \infty$) free boundary problem reveal asymptotic information on the optimal design problem \eqref{eqMin_ps} for $p$ large. Hence, motivated by formal considerations, we consider the following limiting configuration:

\begin{equation}\tag{{\bf $\mathfrak{P}_{\infty}$}}\label{eqMax_infty}
   \displaystyle  \mathfrak{P}_{\infty}[\alpha] \coloneqq  \max\left\{ \int_{\partial \Omega} gv\,d\H \suchthat v \in W^{1, \infty}\left(\Omega\right),\,\,\,  \|\nabla v\|_{L^{\infty}(\Omega)}\leq 1 \,\,\text{and} \,\, \Leb(\{v>0\}) \leq \alpha\right\}.
\end{equation}

    This problem might be called an ``$L^{\infty}-$variational problem'' because of the $L^{\infty}-$bound on the gradient, and because it arises as the limit for the constrained optimization problem \eqref{eqMin_ps} as $p\to \infty$.

    Under the assumption that $g$ is such that $\displaystyle \int_{\partial \Omega} g(x)d\H >0$, we prove here that any sequence of minimizers $u_p$ to \eqref{eqMin_ps} converges, up to a subsequence, to a solution $u_{\infty}$ of the limiting problem \eqref{eqMax_infty}.

\begin{theorem}\label{teo.1.intro} Assume that $\displaystyle \int_{\partial \Omega} g(x)d\H >0$ and let $u_p$ be a minimizer to \eqref{eqMin_ps}. Then, up to a subsequence,
$$
  u_p \to u_{\infty} \quad  \mbox{as} \quad p \to \infty,
$$
uniformly in $\overline{\Omega}$ and weakly in $W^{1, q}(\Omega)$ for all $1 < q < \infty$. Furthermore, such a limit is an extremal of \eqref{eqMax_infty}.
\end{theorem}

    Furthermore, we find that $u_{\infty}$ verifies $-\Delta_{\infty} u_{\infty}(x) =  0$ (in the viscosity sense) in the set $\Omega_{\infty} \defeq \{u_{\infty}>0\} \cup \{u_{\infty}<0\}$ (notice that we just have $u_\infty=0$ in $\Omega \setminus \Omega_{\infty}$). We also compute the limit boundary condition.

\begin{theorem}\label{teo.2.intro} A uniform limit of solutions of \eqref{eqMin_ps} fulfils
\begin{equation}\label{eqlim}
F_{\infty}(x, \nabla u_{\infty}, D^2u_{\infty}) \defeq  \left\{
\begin{array}{rclcl}
   -\Delta_{\infty} u_{\infty}(x) & = & 0 & \mbox{in} & \{u_{\infty}> 0\} \cup \{u_{\infty}< 0\},  \\[5pt]
     u_{\infty}(x) & = & 0 & \mbox{in} & \Omega \setminus ( \{u_{\infty}> 0\} \cup \{u_{\infty}< 0\}),  \\[5pt]
  H(x, \nabla u) & = & 0 & \mbox{on} & \partial \Omega
  \cap\partial(\{u_{\infty}> 0\} \cup \{u_{\infty}< 0\}),
   \end{array}
\right.
\end{equation}
in the viscosity sense, where
$$
\begin{array}{ll}
 H(x, \nabla u) \defeq
  \left\{
\begin{array}{rcl}
   \min\left\{|\nabla u|-1, \frac{\partial u}{\partial \eta}\right\}  & \mbox{if} & x \in \{g>0\}, \\[5pt]
   \max\left\{1-|\nabla u|, \frac{\partial u}{\partial \eta}\right\}  & \mbox{if} & x \in \{g<0\}, \\[5pt]
   \frac{\partial u}{\partial \eta}&\mbox{if} & x \in \{g=0\}.
\end{array}
\right.
\end{array}
$$
\end{theorem}

    In contrast with the limit optimal design problems with Dirichlet boundary condition studied previously in \cite{daSR,RT}, see also \cite{RW}, this Neumann counterpart does not have a point-wise boundary condition. Indeed, the limiting boundary condition depends on the sign of $g$ and must be understood in a more general/appropriated sense in the framework of viscosity solutions theory (see Definition \ref{DefViscSol}), thus losing its variational character when compared to original problem \eqref{eqMin_ps}.

\subsection*{Monge-Kantorovich type problems}

    Let us recall that optimal transport theory is a longstanding research subject that nowadays still attracts growing attention due to its wide variety of emerging applications (cf. \cite{Amb,ACBBV,Ev99,EvGang,G-AMPR,G-AMPR3,IMRT,MRT3,Prat,Vil03,Vil09} and references therein). Historically, these studies began with Gaspard Monge's classical works and were ``rediscovered'' by Kantorovich in the context of economics (matching problems). They also constitute important topics within the context of probability (the Wasserstein metric), analysis (functional inequalities), geometry (Monge-Amp\`{e}re type equations) and PDEs (rates of decay for nonlinear evolution equations) just to name a few.

    Now, we will briefly present some well-known results related to the Monge-Kantorovich
    mass transport theory which will be used throughout the article (cf. \cite{Amb,ACBBV,Ev99,EvGang,Vil03} and \cite{Vil09} for some surveys).
    Let $\mu \in \mathcal{M}(\mathrm{X})$ and $\nu \in \mathcal{M}(\mathrm{Y})$ be Radon measures. We say that $T_{\sharp}\mu = \nu$, i.e., $T: \mathrm{X} \to \mathrm{Y}$ transports $\mu$ onto $\nu$ if
    $$
    \nu(\mathrm{B}) = \mu \left(T^{-1}(\mathrm{B})\right)
    $$
    for every Borel set $\mathrm{B} \subset \mathrm{Y}$.
    We also say that such a map $T$ is a measure-preserving map with respect to $(\mu, \nu)$ or that $T$ pushes $\mu$ forward to $\nu$. Finally, we define the following class
    $$
    T(\mu, \nu)\defeq \left\{T: \mathrm{X} \to \mathrm{Y}\colon T_{\sharp}\mu = \nu\right\}.
    $$

    Let us recall that the \textit{Monge problem}, associated with the measures $\mu$ and $\nu$, consist of finding a map $T^{\ast} \in T(\mu, \nu)$ which minimizes the functional (transportation cost)
\begin{equation}\label{EqOptTrans}
  \displaystyle \inf_{T(\mu, \nu)} \int |x-T(x)|d\mu(x) \qquad \left(\inf_{T(\mu, \nu)} \int c(x, T(x))d\mu(x)\right).
\end{equation}
    Notice that if $\mu$ and $\nu$ are absolutely continuous with respect to the Lebesgue measure, $\mu = f_1\Leb \llcorner \mathrm{X}$ and $\nu = f_2\Leb \llcorner \mathrm{Y}$, then there exists such an optimal map $T: \mathrm{X} \to \mathrm{Y}$. A map $T^{\ast} \in T(\mu, \nu)$ fulfilling \eqref{EqOptTrans} is denoted an \textit{optimal transport map} of $\mu$ to $\nu$.

    The Monge problem is, in general, ill-posed. To overcome such an obstacle, in the early forties, Kantorovich in \cite{Kant} proposed a relaxed version of the Monge problem, as well as introduced a dual variational formulation: Let $\pi_t(x, y) \defeq (1-t)x + ty$ and $\gamma \in \mathcal{M}(\mathrm{X}, \mathrm{Y})$ be a Radon measure. The projections $\text{proj}_x(\gamma) \defeq {\pi_0}_{\sharp} \gamma$ and $\text{proj}_y(\gamma) \defeq {\pi_1}_{\sharp} \gamma$ are denoted marginals of $\gamma$. Under these concepts, the \textit{Monge-Kantorovich problem} (cf. \cite{Kant} and \cite{Prat}), consists of considering the following minimization problem:
\begin{equation}\label{EqMonKant}
  \displaystyle \min \left\{\int_{\mathrm{X} \times \mathrm{Y}} |x-y|d\gamma(x, y): \gamma \in \Pi(\mu, \nu)\right\},
\end{equation}
    where
    $$
    \Pi(\mu, \nu) \defeq \left\{\gamma \in \mathcal{M}(\mathrm{X}, \mathrm{Y}): \,\, \text{proj}_x(\gamma) \defeq {\pi_0}_{\sharp} \gamma\,\,\, \text{and}\,\,\, \text{proj}_y(\gamma) \defeq {\pi_1}_{\sharp} \gamma\right\}.
    $$
    The elements in $\Pi(\mu, \nu)$ are denoted \textit{transport plans} between $\mu$ and $\nu$, and a minimizer to \eqref{EqMonKant} an \textit{optimal transport plan}. It is worth stress that a minimizer to \eqref{EqMonKant} always exists.

    Another important peace of information is that the Monge-Kantorovich problem admits the following dual formulation, known as the {\it Kantorovich-Rubinstein theorem}, \cite[Theorem 1.14]{Vil03} in the literature: The following duality holds true
\begin{equation}\label{EqDualMonKant}
  \displaystyle \min \left\{\int_{\mathrm{X} \times \mathrm{Y}} |x-y|d\gamma(x, y): \gamma \in \Pi(\mu, \nu)\right\} = \max\left\{\int_{\mathrm{X}} ud(\mu-\nu):\,\,\, u \in 1-\text{Lip}(\overline{\mathrm{X}})\right\},
\end{equation}
    where
    $
    \displaystyle 1-\text{Lip}(\overline{\mathrm{X}}) \defeq \Big\{u: \overline{\mathrm{X}} \to \R: \,\,\,   \sup_{x, y \in \overline{\mathrm{X}}, \, {x\neq y}} \frac{|u(x)-u(y)|}{|x-y|}\leq 1  \Big\}$.
    Maximizers of \eqref{EqDualMonKant} are called \textit{Kantorovich potentials}.

    Regarding the $\infty-$Neumann problem, the limit maximization problem \eqref{EqLMP} is also obtained by considering a dual formulation of the well-known \textit{Monge-Kantorovich mass transfer problem} for the measures $\mu = g^{+}\H \llcorner \partial \Omega$  and $\nu = g^{-}\H \llcorner \partial \Omega$ supported on $\partial \Omega$, where such measures must fulfil the mass transfer compatibility condition $\mu(\partial \Omega) = \nu(\partial \Omega)$ (cf. \cite{Amb} and compare with \cite[Theorem 1.1]{G-AMPR}).

    Our next result enables us to find a Kantorovich potential for the optimal mass transport problem via uniform convergence of a subsequence of the family of solutions to \eqref{eqMin_ps}.

\begin{theorem}\label{teo.3.intro} Let $g\geq 0$. There exists a non-negative measure $\nu = \nu_\infty$ such that a uniform limit of solutions of \eqref{eqMin_ps}, i.e., $\displaystyle u_{\infty}(x) = \lim_{p \to \infty} u_p(x)$, is a Kantorovich potential for the optimal mass transport problem between $\mu = g\H \llcorner \partial \Omega$ and $\nu_\infty$ (supported on the limiting free boundary).
\end{theorem}

Finally,  this limit gives the maximum possible transport cost
    between $\mu = g\H \llcorner \partial \Omega$ and any nonnegative measure $\nu$
    with transport set of measure less or equal than $\alpha$.
    Notice that the infimum of such costs is zero (just consider $\nu_n$ a sequence
    of measures converging to $g\H \llcorner \partial \Omega$ with supports converging to $\partial \Omega$).

\begin{theorem}\label{teo.5.intro} Suppose that the assumptions of Theorem \ref{teo.3.intro} are in force. Then,
$$
   \displaystyle  \int_{\partial \Omega} u_\infty g d\H = \max_{\nu \in \mathcal{M}(\Omega), \,\omega \in 1-\text{Lip}(\overline{\Omega}), \atop{\Leb (\mathrm{T}(\omega))\leq \alpha}} \left\{\int_{\overline{\Omega}} \omega d(\mu-\nu) \right\}.
$$
\end{theorem}

\medskip

    Our manuscript is organized as follows: in Section \ref{Prelim} we collect some preliminary results
    that will be used throughout the article and analyze the problem for a finite (fixed) $p$. In Section \ref{Sec3} we show how to pass to the limit as $p\to \infty$. Section \ref{Sec4} is devoted to explain how our limiting free boundary optimization problem links with the Monge-Kantorovich mass transfer problem. Finally, in Section \ref{Sec5} we include some examples in which limit solutions can be computed explicitly.

%%%%%%%%%%%%%%%%%%%%%%%%%%%%%%%%%%%%%%%%%%%%%%%%%%%%%%%%%%%%%%%%%%%%%%%%%%

\section{Analysis for finite $p$} \label{Prelim}

    Throughout this manuscript $\Omega \subset \R^N$ will denote an open and bounded
    domain with Lipschitz boundary with a unitary outward normal vector field
    $\eta\colon \partial \Omega\to \mathbb{S}^{N-1}$ that is defined for $\H-$almost
    every point of $\partial \Omega$, where $\H$ states the standard
    $(N-1)-$dimensional Hausdorff measure.

    Now we specify the different notions of solutions which we will use throughout this
    article. For a fixed value of $N<p< \infty$ we consider weak solutions.
    On the other hand, in the limiting setting, as $p \to \infty$,
    we will use the concept of viscosity solutions.

    \begin{definition}[{ Weak solution}]
        Let $p>N.$ A $u \in W^{1, p}(\Omega)$ is said a weak solution to
        \eqref{eqp} if there holds
        \[
            \displaystyle
            \int_{\Omega \setminus \{u=0\}} |\nabla u|^{p-2}\nabla u \cdot \nabla \phi dx
            = \int_{\partial \Omega} g \phi\, d\H
        \]
        for every $\phi \in W^{1, p}(\Omega \setminus \{u=0\})$
        with $\phi \equiv 0$ in $ \{u=0\}$.
    \end{definition}

    Now, our aim is to show that there is
    a minimizer of the functional
    \[
       \mathcal{J}_p[v]
       \defeq \frac{1}{p}\int_{\Omega}
       |\nabla v|^pdx- \int_{\partial \Omega} gv \, d\H
    \]
    over
    \[
        \mathbb{K}^p_{\alpha} \defeq \left\{v \in W^{1, p}\left(\Omega\right)\colon
            \Leb(\{v>0\}) \leq \alpha\right\}.
    \]
    Note that, following \cite{G-AMPR}, we can show that any minimizer
    of $\mathcal{J}_p[\cdot]$ over $\mathbb{K}_{\alpha}^p$ is a weak
    solution to \eqref{eqp}.

     Let us recall an important inequality.

     \begin{theorem}[Morrey's inequality]\label{MorIneq}
        Let $N<p\leq \infty$ and $\Omega \subset \R^n$ be a regular domain. Then for all $u \in W^{1, p}(\Omega)$ such that $\{u=0\}\neq\emptyset$, there exists a constant $C(N, p, \Omega)>0$ such
        that
       \[
            \|u\|_{C^{0, 1-\frac{N}{p}}(\Omega)}
                \leq C(N, p, \Omega)\|\nabla u\|_{L^{p}(\Omega)},
       \]
       where the constant $C(N, p, \Omega)>0$ can be assumed uniform
       in $p$.
    \end{theorem}

    We now prove existence of minimizers for our minimization problem.
    Taking into account that we are interested in the asymptotic limit as
    $p \to \infty$, we will assume that $p>N$.

    \begin{theorem}[{Existence of minimizers}]\label{ThmExistMin}
        Let $p>N,$ $g\in L^1(\partial\Omega)$ be such that
        \[
            \int_{\partial \Omega} g \,d\H >0
        \]
        and $0< \alpha < \Leb(\Omega)$, fixed.
        Then there is at least one function $u_p\in\mathbb{K}_\alpha^p$
        solving
        \[
            \mathcal{J}_p[u_p]=\min\left\{\mathcal{J}_p[v]\colon
            v\in\mathbb{K}_\alpha^p \right\}.
        \]
        In addition, if $u$ is minimizer
            of $\mathcal{J}_p[\cdot]$ over $\mathbb{K}_{\alpha}^p$  then
        \[
            \Leb(\{u>0\})=\alpha.
        \]
    \end{theorem}

    \begin{proof}
        First, we claim that
        \begin{equation}\label{Eq1ThmEx}
            \inf
            \left\{ \mathcal{J}_p[v]\colon v\in \mathbb{K}^p_{\alpha}
            \right\} <0.
        \end{equation}
        To see this, we take $a>0$ such that
        $\Leb(\{x\in\overline{\Omega}\colon \text{dist}(x,\partial\Omega)\le a\})=\alpha,$
        $\varepsilon>0,$ and $v\in W^{1,p}(\Omega)$ the weak solution of
        \[
            \begin{cases}
	            -\Delta_p u=0 &\text{in }
	            \qquad \Omega_a=\{x\in\Omega\colon \text{dist}(x,\partial\Omega)\le a\},\\
	            \qquad u=\varepsilon & \text{on }\qquad \partial \Omega,\\
	            \qquad u=0& \text{on } \qquad \{x\in\overline{\Omega}\colon \text{dist}(x,\partial\Omega)
	            = a\}.
	        \end{cases}
        \]
        Then
        \[
            \psi(x)\defeq
            \begin{cases}
                u(x) &\text{if } x\in\Omega_a,\\
                0 &\text{if } x\in\Omega\setminus\Omega_a,
            \end{cases}
        \]
        belongs to $ \mathbb{K}^p_{\alpha}$ and $\mathcal{J}_p[\psi]<0$ provided
        $\varepsilon$ is small enough. Thus \eqref{Eq1ThmEx} follows.

        \medskip

        Now, we consider a minimizing sequence for \eqref{eqMin_ps},
        i.e., $(u_j)_{j \in \mathbb{N}} \subset W^{1,p}(\Omega)$
        such that
        \[
             \Leb(\{u_j>0\}) \leq \alpha \quad
             \text{and} \quad  \mathcal{J}_p[u_j] \searrow
             \inf
            \left\{ \mathcal{J}_p[v]\colon v\in \mathbb{K}^p_{\alpha}
            \right\}.
        \]

        Next, we assert that we can assume that for each $j \in \mathbb{N}$ there exists
        at least one $x_j \in \overline{\Omega}$ such that $u_j(x_j) = 0.$
        To verify this claim, first note that $\{u_j>0\} \neq \overline{\Omega}$.
        On the other hand, if $\{u_j>0\} \neq \emptyset$ then
        $u_j$ must change sign and then there exists
        $x_j \in \overline{\Omega}$ such that $u_j(x_j) = 0$.
        Now, if $\{u_j<0\}  =  \overline{\Omega}$, then for each $j \in \mathbb{N}$
        we could select an $\varepsilon_j >0$ such that $
        \Leb(\{u_j+\varepsilon_j>0\}) \leq \alpha$ with
        $\{u_j+\varepsilon_j\geq 0\} \cap \overline{\Omega} \neq \emptyset$.
        From our assumption on $g$ we get
        $$
        \begin{array}{rcl}
          \inf \left\{ \mathcal{J}_p[v]\colon v\in \mathbb{K}^p_{\alpha}
            \right\}  & \leq & \mathcal{J}_p[u_j+\varepsilon_j]\\[7pt]
           & = &  \displaystyle \mathcal{J}_p[u_j] -
            \varepsilon_j \int_{\partial \Omega} g(x)d\H \\[7pt]
            & < & \mathcal{J}_p[u_j]
             \to  \inf \left\{ \mathcal{J}_p[v]\colon v\in \mathbb{K}^p_{\alpha}
            \right\},
        \end{array}
        $$
        and then we can just take $u_j + \epsilon_j$ as our minimizing sequence.
        Notice that there exists at least one point $x_j \in \overline{\Omega}$
        such that $u_j(x_j) +\epsilon_j = 0$. Hence, our claim is proved.

        \medskip

        In what follows, we will still call $u_j$ the minimizing sequence with
        $u_j(x_j) = 0$. Next, using Morrey's inequality, we get
        \begin{align*}
                \int_{\partial \Omega} gu_jd\H&
                \leq   \int_{\partial \Omega} |g(x)||u_j(x)-u_j(x_j)|\,d\H \\
                  & \leq  C(N, p, \Omega)
                  \|\nabla u_j\|_{L^{p}(\Omega)}
                  \int_{\partial \Omega} |g(x)| |x-x_j|^{1-\frac{N}{p}}\,d\H\\
                    & \leq   C(N, p, \Omega)
                    \|g\|_{L^1(\partial \Omega)}
                    \text{diam}(\Omega)^{1-\frac{N}{p}}
                    \|\nabla u_j\|_{L^{p}(\Omega)}.
        \end{align*}
        Therefore,
            \begin{equation}\label{eq2.2}
                \mathcal{J}_p[u_j] \geq \frac{1}{p}
                \|\nabla u_j\|^p_{L^{p}(\Omega)} -
                C\left(N, p, \|g\|_{L^{1}(\partial \Omega)}, \Omega \right)
                \|\nabla u_j\|_{L^{p}(\Omega)}.
        \end{equation}

        We now claim that $(u_j)_{j \in \mathbb{N}}$ must fulfil
            \[
                \|\nabla u_j\|_{L^p(\Omega)} \leq C(N, p, \Omega)
            \]
            uniformly in $p.$
            Otherwise, if for some subsequence
            $\|\nabla u_{j_k}\|_{L^p(\Omega)} \to \infty$ as $k \to \infty.$
            Then we would conclude from \eqref{eq2.2} that
            \[
                \mathcal{J}_p[u_{j_k}] \to \infty,
            \]
            which contradicts \eqref{Eq1ThmEx}.

            Furthermore, for $x_j \in \overline{\Omega}$ such that $u_j(x_j) = 0$
            (whose existence we already assured) we obtain
            \[
                |u_j(x)|  =  |u_j(x) - u_j(x_j)|  \leq
                C(N, p, \Omega) \|\nabla u_j\|_{L^p(\Omega)}|x-x_j|^{1-\frac{N}{p}}
                \leq  C(N, p, \Omega)\text{diam}(\Omega)^{1-\frac{N}{p}}.
            \]
            Therefore,
            \[
                \|u_j\|_{L^{\infty}(\Omega)}\leq C(N, p, \Omega).
            \] Hence, $(u_j)_{j \in \mathbb{N}}$ is uniformly bounded and equicontinuous.
            From compact embedding, converges (up to a subsequence) to a function
            $u_p$ strongly in $C^{0, 1-\frac{N}{p}}(\Omega)$. Thus, from the previous
            convergence we obtain
            \[
                \Leb(\{u_p>0\}) \leq \liminf_{j \to \infty} \Leb(\{u_j>0\}) \leq \alpha,\quad -\int_{\partial \Omega} gu_jd\H \to -
                \int_{\partial \Omega} gu_p\,d\H,
            \]
            and
            \[
                 \int_{\Omega} |\nabla u_p|^pdx \leq \liminf_{j \to \infty}
                 \int_{\Omega} |\nabla u_j|^pdx.
            \]
            Therefore, we conclude that
            \[
                \mathcal{J}_p[u_p] \leq \liminf_{j \to \infty} \mathcal{J}_p[u_j],
            \]
            which assures that $u_p$ is a minimizer. Observe that \eqref{Eq1ThmEx}
            $u_p\not \equiv0.$

    \medskip

        Finally, we show that if $u$  is minimizer
            of $\mathcal{J}_p[\cdot]$ over $\mathbb{K}_{\alpha}^p$   then
            \begin{equation}\label{eq:lba}
	            \Leb(\{u>0\})=\alpha.
            \end{equation}
            The proof is by contradiction.
            Suppose that there exists a minimizer $u$ and a constant
            $0< \varepsilon \ll 1$ such that
            \[
                 \Leb(\{u>0\}) = \alpha - \varepsilon.
            \]
            Notice that, arguing as before, we can show that $u\not\equiv0$ and $\{u>0\}\neq \emptyset $.

            Now, for $x_0 \in \partial \{u>0\}\cap\Omega$ fixed, select
            \[
                    0<r< \min\left\{\frac{1}{2}\dist(x_0, \partial \Omega), \,
                    \sqrt[n]{\frac{\varepsilon}{2\omega_N}}\right\},
            \]
            where $\omega_N  = \Leb(B_1(0))$.
            Next, we solve the following minimization problem:
            \[
                \min \left\{\frac{1}{p}\int_{B_r(x_0)} |\nabla v|^pdx\colon v \in W^{1, p}
                \left(B_r(x_0)\right),\,\,  v =  u \,\,\text{on} \,\,\partial B_r(x_0)\right\}.
            \]
            Such minimizers, let us call them $v_0$, are $p$-harmonic functions in
            $B_r(x_0)$. Moreover, notice that $u$ competes with $v_0$ in
            the minimization problem in $u$, that is
            \begin{equation}\label{eq2.3}
                 \frac{1}{p}\int_{B_r(x_0)} |\nabla v_0|^pdx
                 < \frac{1}{p}\int_{B_r(x_0)} |\nabla u|^pdx,
            \end{equation}
            where the strict inequality comes from the fact that $u$ is $p$-harmonic in
            $\{u\neq 0\}\cap B_r(x_0)$, but it is not $p-$harmonic
            across the free boundary. Now, setting
            \begin{equation}\label{eq2.4}
                \psi(x)\defeq
                    \begin{cases}
                        v_0(x) & \text{in }  \quad B_r(x_0),\\
                        u(x) & \text{in }  \quad \Omega \setminus B_r(x_0),	
                    \end{cases}
             \end{equation}
            we obtain a profile such that $v \in W^{1, p}(\Omega)$ and
                \begin{align*}
                    \Leb(\{\psi>0\}) & \leq   \Leb(\{u>0\}\setminus B_r(x_0))
                    +  \Leb(\{u>0\} \cap B_r(x_0))\\
                    & \leq   \Leb(\{u>0\}) +  \Leb(B_r(x_0)) \\
                    & <  (\alpha-\varepsilon) + \varepsilon =  \alpha.
                \end{align*}
            Finally, using \eqref{eq2.3} and \eqref{eq2.4} we conclude that
            \[
                \mathcal{J}_p[\psi] < \mathcal{J}_p[u]
                = \inf \left\{ \mathcal{J}_p[v] \colon v\in
                \mathbb{K}_\alpha^p
                \right\},
            \]
            contradicting the minimality of $u$. This completes the proof.
    \end{proof}

	\begin{remark}
		If $u$ and $v$ are two minimizers of $\mathcal{J}_p[\cdot]$ such that
		$\Leb(\{u+v>0\})\le \alpha$
		then $u\equiv v.$ This is due to the fact that $\mathcal{J}_p[\cdot]$ is strictly convex.
	\end{remark}

    \begin{remark}[Assumption on the boundary datum] \label{rem24}
        In this part we will discuss about the assumption on $g$.
        Remind that we have assumed the condition:
        \[
            \int_{\partial \Omega} g\,d\H>0.
        \]
        However, we could also consider two other possibilities:
        \begin{enumerate}
            \item $\displaystyle \int_{\partial \Omega} g(x)d\H=0$.
               In this case, our minimization problem reduces to
                \[
                    \inf \mathcal{J}_p[v] =
                    \inf\left\{ \mathcal{J}_p[v]\colon v \in \mathbb{K}^p_{\alpha}\right\}.
                \]
                In fact, for any constant $c>0$ and any admissible function
                $u \in \mathbb{K}^p_{\alpha}$ we have that
                $v= u-c \in \mathbb{K}^p_{\alpha}$ and
                \[
                    \mathcal{J}_p[v] = \mathcal{J}_p[u] - c\int_{\partial \Omega} gd\H
                    = \mathcal{J}_p[u].
                \]
                Therefore, in this case the volume constraint does not play any
                significant role in the minimization problem (compare with \cite{G-AMPR}).

            \item $\displaystyle \int_{\partial \Omega} g\,d\H<0$.
                In this case, by consider any sequence $0<a_k \to \infty$
                as $k \to \infty$, the constant functions
                $u_k=-a_k \in \mathbb{K}^p_{\alpha}$ satisfy
                \[
                    \mathcal{J}_p[u_k] = a_k\int_{\partial \Omega} g\,d\H \to -\infty
                    \quad \text{as } k \to \infty,
                \]
                which implies that our minimization problem does not admit a minimizer.
        \end{enumerate}
    \end{remark}

	\begin{remark} \label{rem.2.6}
It is straightforward to verify that when the boundary datum $g$ is a non-negative
    function, then any minimizer $u_0$ to \eqref{eqMin_ps} will also be non-negative in the whole
    $\overline{\Omega}$. This remark will be crucial in the symmetry results and in the optimal transportation argument.
	\end{remark}

    \subsection*{A spherical symmetrization result}

        Next, we will look at our optimization problem
        when the domain is a ball, $\Omega = B_1(0)$, and $g$ is
        spherically symmetric and strictly decreasing with respect to some axis. For that purpose, an essential
        tool is played by the \textit{spherical symmetrization}.

        Given a measurable set $\mathcal{E}\subset\R^N$, the spherical symmetrization
        $\mathcal{E}^{\ast}$ of $\mathcal{E}$ with respect to an axis given by a unit
        vector $e_k$ is constructed as follows:
        For each positive number $r$, take the intersection
        $\mathcal{E}\cap \partial B_r(0)$ and replace it by the spherical portion
        of the same $\mathcal{H}^{N-1}-$measure and center $re_k.$ The union of these
        caps is  $\mathcal{E}^{\ast}.$
        Now, the spherical symmetrization $u^{\ast}$ of a measurable function
        $u\colon \Omega \to \R$ is constructed by symmetrizing the super-level sets
        so that, for all $t$
        $$
            \{u^{\ast}\geq t\} = \{u\geq t\}^{\ast}.
        $$
        We recommend to the reader references \cite{Kawohl} and \cite{Sperner}
        for more details.
       We will use the following result.

        \begin{theorem}\label{reluu*}\hfill
            \begin{enumerate}
                \item[a)]\label{Stet1}  Let $u \in W^{1, p}(B_1(0))$ be non-negative.
                     Then, $u^{\ast} \in W^{1, p}(B_1(0)),$ and
                    \[
                        \int_{B_1(0)} |u^{\ast}|^p\, dx =
                        \int_{B_1(0)} |u|^p\, dx,\quad
                        \text{ and }
                        \quad
                         \int_{B_1(0)} |\nabla u^{\ast}|^p\, dx \leq
                        \int_{B_1(0)} |\nabla u|^p\, dx.
                    \]
                \item[b)]\label{Stet2} If $u$ is a non-negative mensurable function in
                $\overline{B_1(0)}$ and $v$ is a non-negative mensurable function in
                $\partial B_1(0)$ then
                \begin{equation}\label{eq:P1}
	                \int_{\partial B_1(0)} uv\,d\H
                \leq \int_{\partial B_1(0)} u^{\ast}v^{\ast}\,d\H.
                \end{equation}

            \end{enumerate}
        \end{theorem}

        \begin{proof}
            We first show {\it (a)}.
            By \cite[(C) page 22]{Kawohl},
            \begin{equation}\label{eq:aux1}
	            \int_{B_1(0)} |f|^p dx=\int_{B_1(0)} |f^\ast|^p dx.
            \end{equation}
            for any non-negative function $f\in L^p(B_1(0)).$ Therefore,
            we only need
            to show that
            if $u \in W^{1, p}(B_1(0))$ is non-negative then
            \[
	             \int_{B_1(0)} |\nabla u^{\ast}|^p\, dx \leq
                        \int_{B_1(0)} |\nabla u|^p\, dx.
            \]

            In \cite{Sperner}, the author show that
            if $v\in C^\infty(\R^N)$ and is non-negative then
            \begin{equation}\label{eq:aux2}
	            \int_{B_1(0)}|\nabla v^\ast|^p \, dx\le\int_{B_1(0)}|\nabla v|^p \, dx.
            \end{equation}
            Whereas in \cite[(M7) page 21]{Kawohl}, it is proven that
            \begin{equation}\label{eq:aux3}
	             \|f^\ast-g^\ast\|_{L^1(B_1(0))}\le\|f-g\|_{L^1(B_1(0))}
            \end{equation}
            for every non-negative functions $f, g\in L^1(B_1(0)).$

            Given a non-negative function $u \in W^{1, p}(B_1(0)),$ we take
            \[
                \bar{u}(x)=
                    \begin{cases}
                        u(x) &\text{ if } \quad x\in B_1(0),\\
                        0 &\text{ if } \quad x\in \mathbb{R}^N\setminus B_1(0),
                    \end{cases}
            \]
            and set $v_n=\rho_n\star \bar{u}$ (where $\rho_n$ is a sequence of mollifiers).
            Then $v_n\in C^{\infty}(\mathbb{R}^N)$ is nonnegative and $v_n\to u$
            strongly in
            $W^{1,p}(\Omega).$ Moreover, using  \eqref{eq:aux1}, \eqref{eq:aux2}, and
            \eqref{eq:aux3}, we have that $v_n^\ast\to u^\ast\ $
            weakly in $W^{1,p}(\Omega).$ Therefore
            \[
	            \int_{B_1(0)}|\nabla u^\ast|^p \, dx\le
	            \liminf_{n\to\infty}\int_{B_1(0)}|\nabla v_n^\ast|^p \, dx\le
	            \lim_{n\to\infty}\int_{B_1(0)}|\nabla v_n|^p \, dx
	            =\int_{B_1(0)}|\nabla u|^p \, dx.
            \]

            To finish the proof, we prove {\it (b).} In first step, we show that
            \eqref{eq:P1} holds for characteristic function. Let $A\subset
            \overline{B_1(0)}$ and $B\subset\partial B_1(0)$ be two mensurable sets and
            $u(x)=\chi_A(x)$ and $v(x)=\chi_B(x).$ Observe that, by definition,
            $u^\ast(x)=\chi_{A^\ast}(x)$ and $v^\star(x)=\chi_{B^\ast}(x)$ and
            $A^\ast\cap\partial B_1(0)\subseteq B^\ast$ or
            $B^\ast\subseteq A^\ast\cap\partial B_1(0).$ Thus
            \[
                u(x)v(x)=
                 \begin{cases}
                    1 &\text{if } \quad x\in A\cap B,\\
                    0 &\text{if } \quad x\in \R^N\setminus A\cap B,
                 \end{cases}
                 \quad \text{ and } \quad
                u^\ast(x)v^\ast(x)=
                 \begin{cases}
                    u^\ast(x) &\text{if } \quad A^\ast\cap\partial B_1(0)\subseteq B^\ast,\\
                    v^\ast(x) &\text{if } \quad B^\ast\subseteq A^\ast\cap\partial B_1(0),
                 \end{cases}
            \]
            Then,
            \begin{align*}
                 \int_{\partial B_1(0)} uv\,d\H&=
                 \H(A\cap B)\\
                 &\le
                 \begin{cases}
                    \H (A\cap\partial B_1(0)) \\
                   \H (B)
                 \end{cases}\\
                 &=
                 \begin{cases}
                    \H (A^\ast\cap\partial B_1(0)) \\
                   \H (B^\ast)
                 \end{cases}\\
                 &=  \int_{\partial B_1(0)} u^\ast v^\ast\,d\H.
            \end{align*}

            Thus, it is easy to see that \eqref{eq:P1} holds for non-negative steps
            function. Finally, as any measurable function can be approximate by steps
            functions, we can prove the assertion by an approximation argument.
        \end{proof}

        \begin{remark}
        \label{rem-symm}
        Notice that, if $v=v^{\ast} \geq 0$ is spherically strictly decreasing, then equality in
        {\it (b)},
        $$
	                \int_{\partial B_1(0)} uv^*\,d\H
                \leq \int_{\partial B_1(0)} u^{\ast}v^{\ast}\,d\H,
                $$
                for a non-negative $u$
                implies that also $u$ is spherically symmetric, $u=u^{\ast}$.
                In fact, we have
                $$
                \begin{array}{rcl}
                \displaystyle
                \int_{\partial B_1(0)}
uv^{\ast}  \,d\H & = &  \displaystyle \int_{\partial B_1(0)} \int_0^\infty  \int_0^\infty
\chi_{\{u(x)>s\}} \chi_{\{v^*(x)>t\}} \, ds\, dt \,d\H\\[10pt]
& = & \displaystyle \int_0^\infty  \int_0^\infty
\H (\{u(x)>s\} \cap \{v^*(x)>t\}) \, ds\, dt\\[10pt]
 & =& \displaystyle \int_0^\infty  \int_0^\infty
\H (\{u^*(x)>s\} \cap \{v^*(x)>t\}) \, ds\, dt\\[10pt]
 & = & \displaystyle \int_{\partial B_1(0)} u^{\ast}v^{\ast}  \,d\H.
\end{array}
                $$
                Therefore, $u$ and $v^{\ast}$ have the same family of level sets, and hence $u=u^*$. Note that we are using here that when $v=v^{\ast}$ is strictly spherically decreasing its family of level sets covers the whole family of spherical caps, from $\{ e_k \}$ to the whole $\partial B_1(0)$.
        \end{remark}

        Finally, we prove our symmetry result. This ends the proof of Theorem \ref{ThmExistMin.intro}.

        \begin{theorem}\label{SymmetThm}
            Let $\Omega=B_1(0)$ and $u_p$ be  a minimizers of $
		    \mathcal{J}_p[\cdot]$ over $\mathbb{K}_{\alpha}^p.$
		    Suppose that $0\leq g = g^{\ast}$.
		    Then, there is a minimizer, $u_p$, that
		    is spherically symmetric.
		
		    In addition, when $0\leq g=g^\ast$ is spherically strictly decreasing, every
		    minimizer is spherically symmetric on $\partial B_1(0)$.
        \end{theorem}

        \begin{proof}
            Theorem \ref{ThmExistMin} assures that there exists a profile
            $u_p\in W^{1, p}(\Omega)$ such that
            $$
                \Leb(\{u_p>0\})=\alpha \qquad\text{and}\qquad
                \mathcal{J}_p[u_p]=\inf\left\{\mathcal{J}_p[v]\colon
                v\in\mathbb{K}_\alpha^p \right\}.
            $$
            Now, let $u_p^{\ast}$ be the spherical symmetrization of $u_p$.
            Notice that $u_p^{\ast}$ is an admissible profile in
            the optimization process of $\mathcal{J}_p[\cdot]$. In fact, by Remark \ref{rem.2.6}, since $g \geq 0$ then
            $u_p \geq 0$ and therefore one can apply the results in Theorem \ref{reluu*} to obtain that
            $$
                u^{\ast} \in W^{1, p}(\Omega), \,\,\,\Leb(\{u_p^{\ast}>0\})
                = \Leb(\{u_p>0\})= \alpha \,\,\, \text{and} \,\,\,
                \displaystyle -\int_{\partial \Omega} ug\,dx \geq
                -\int_{\partial \Omega} u^{\ast}g^{\ast}\,dx =
                 -\int_{\partial \Omega} u^{\ast}g\,dx.
            $$
            Hence, once again by Theorem \ref{reluu*},
            $$
                \inf\left\{\mathcal{J}_p[v]\colon
                      v\in\mathbb{K}_\alpha^p \right\}
                      \leq \mathcal{J}_p[u_p^{\ast}]\leq
                      \mathcal{J}_p[u_p]=\inf\left\{\mathcal{J}_p[v]\colon
                        v\in\mathbb{K}_\alpha^p \right\}.
            $$
            Therefore,
            \begin{equation}\label{sym}
                \inf\left\{\mathcal{J}_p[v]\colon v\in\mathbb{K}_\alpha^p \right\}
                = \mathcal{J}_p[u_p^{\ast}].
            \end{equation}
            Hence, we conclude the existence of a minimizer that is spherically symmetric.

            Now, let us assume that $0\leq g=g^\ast$ is spherically strictly decreasing
            and let $u$ be a minimizer. From our previous calculations we must have
            $$
	                \int_{\partial B_1(0)} ug^{\ast}\,d\H
                \leq \int_{\partial B_1(0)} u^{\ast}g^{\ast}\,d\H,
                $$
            and then, from Remark \ref{rem-symm}, we obtain that $u=u^{\ast}$ on $\partial B_1(0)$, as we wanted to show.
        \end{proof}

        As a byproduct of this result we obtain that there is a minimizer such that its null set
        $\{u_p =0 \}$ is spherically symmetric.

  \subsection*{Viscosity solutions}

    Let us present a brief introduction to the theory of viscosity solutions for second
    order fully nonlinear elliptic equations. Recall that a continuous function
    $F\colon \overline{\Omega}\times \R^N \times \text{Sym}(N) \to \R$
    is called \textit{degenerate} elliptic if
    $$
        F(x, \xi, \mathrm{X}) \leq F(x, \xi, \mathrm{Y})
        \quad \text{whenever} \quad \mathrm{Y}\leq \mathrm{X} \quad
        \text{in the sense of matrices}.
    $$
    Along this paper we will use:
    \begin{enumerate}
        \item $F(x, \nabla u, D^2 u) = - \nabla u^TD^2u
     \nabla u = -\Delta_{\infty} u$;
        \item $F(x, \nabla u, D^2 u) = - \left[|\nabla u|^{p-2}\tr(D^2 u)
        + (p-2)|\nabla u|^{p-4}\nabla u^TD^2u \nabla u\right]$.
    \end{enumerate}

    Taking into account general boundary data, let us recall the appropriate definition of
    viscosity solutions in our context. Concerning general theory of viscosity solutions
    to fully nonlinear elliptic equations we refer the reader to the surveys
    \cite{Barl,CIL,Ish,IL}.

    \begin{definition}[Viscosity solution]\label{DefViscSol}
        Consider the following boundary value problem:
        \begin{equation}\label{EqViscSol}
        \begin{cases}
	        F(x, \nabla u, D^2 u)  =  0  &\text{in }  \quad  A, \\
            \quad H(x, u, \nabla u)  =  0  & \text{on }  \quad  \partial A,
        \end{cases}
        \end{equation}
        where $F\in C(\overline{A}\times \R^N \times \text{Sym}(N))$ is a
        degenerate elliptic function and
        $H\in C(\partial A\times\mathbb{R}\times\R^N).$
		\begin{enumerate}
  			\item A lower semi-continuous function $u$ is said a viscosity supersolution to
  				\eqref{EqViscSol} if for every $\phi \in C^2(A)$ such that $u-\phi$
  				has a strict minimum at the point $x_0 \in \overline{A}$ with
  				$u(x_0) = \phi(x_0)$ we have:
				\begin{itemize}
  					\item[\checkmark] If $x_0 \in \partial A$ the inequality holds
  						\[
  							\max\left\{F(x_0, \nabla \phi(x_0),
  							D^2 \phi(x_0)), H(x_0, \phi(x_0), \nabla \phi(x_0))\right\}\geq 0.
  						\]
  					\item[\checkmark] if $x_0 \in A$ then we require
  						\[
  							F(x_0, \nabla \phi(x_0), D^2 \phi(x_0))\geq 0.
  						\]
			\end{itemize}

  			\item An upper semi-continuous function $u$ is said a viscosity subsolution to
  				\eqref{EqViscSol} if for every $\phi \in C^2(A)$ such that $u-\phi$ has a
  				strict maximum at the point $x_0 \in \overline{A}$ with $u(x_0) = \phi(x_0)$
  				we have:
				\begin{itemize}
 					 \item[\checkmark] If $x_0 \in \partial A$ the inequality holds
  						\[
  							\min\left\{F(x_0, \nabla \phi(x_0), D^2 \phi(x_0)),
  							H(x_0, \phi(x_0), \nabla \phi(x_0))\right\}\leq 0.
  						\]
  					\item[\checkmark] if $x_0 \in A$ then we require
  						\[
  							F(x_0, \nabla \phi(x_0), D^2 \phi(x_0))\leq 0.
  						\]
				\end{itemize}
		\end{enumerate}

		Finally, a continuous function $u$ is said a viscosity solution to
		\eqref{EqViscSol} if it is simultaneously a viscosity supersolution and
		a viscosity subsolution.

		When $F$ is not continuous we need to consider the lower semicontinous $F_*$, $H_*$
		and upper semicontinous $F^*$, $H^*$ envelopes of $F$ and $H$ respectively.
		In 1. of the previous definition
		we ask for
		  						\[
  							\max\left\{F^*(x_0, \nabla \phi(x_0),
  							D^2 \phi(x_0)), H^*(x_0, \phi(x_0), \nabla \phi(x_0))\right\}\geq 0
  					\quad \mbox{ or }  \quad F^*(x_0, \nabla \phi(x_0), D^2 \phi(x_0))\geq 0.
  						\]
    While in 2. we ask for
	\[
  							\min\left\{F_*(x_0, \nabla \phi(x_0),
  							D^2 \phi(x_0)), H_*(x_0, \phi(x_0), \nabla \phi(x_0))\right\}\leq 0
  					\quad \mbox{ or }  \quad F_*(x_0, \nabla \phi(x_0), D^2 \phi(x_0))\leq 0.
  						\]
	\end{definition}

	From now on we assume that $g\in C(\partial\Omega).$ We will use the following notations:
	\[
   		F_p(x, \xi, \mathrm{X}) \defeq -\left[|\xi|^{p-2}\tr(\mathrm{X})
   		+ (p-2)\langle\mathrm{X}\xi, \xi\rangle\right]
		\qquad \mbox{and} \qquad
		H_p(x, \xi) \defeq |\xi|^{p-2}\langle \xi, \eta(x) \rangle-g(x).
	\]
	Notice that these two functions are continuous (and hence $F^*=F_*=F$ and $H^*=H_* =H$).
	
	\begin{remark}\label{RemMonotB_p}
		 We need to highlight that since $H_p$ is monotone in the
		 variable $\frac{\partial u}{\partial \eta}$, then Definition
	     \ref{DefViscSol} admits a simpler form (cf. \cite{Barl}).
		 To be precise, if $u$ is a viscosity supersolution and
		 $\phi \in C^2(\overline{\Omega})$ is such that $u-\phi$ has
		 a strict minimum at $x_0$ with $u(x_0) = \phi(x_0)$, then
		\begin{itemize}
			\item[\checkmark] If $x_0 \in \Omega$, then
			  $$
				  -\left[\Delta_{\infty} \phi(x_0) + \frac{|\nabla
				  	\phi(x_0)|^2\Delta \phi(x_0)}{p-2}\right] \geq 0.
			  $$
			\item[\checkmark] If $x_0 \in \partial \Omega$, then
			  $$
				  H_p(x_0, \phi(x_0)) \geq 0,
			  $$
		\end{itemize}
		and the opposite inequalities for the case in which $u-\phi$ has
		 a strict maximum at $x_0$.
		
    Observe that the limit boundary condition \eqref{eqlim} does not fulfil such a monotonicity condition and hence to understand sub and super solutions in the viscosity sense at boundary points one needs to take $\min$ or $\max$ between the equation and the boundary condition as in Definition \ref{DefViscSol}.
	\end{remark}

	The next result gives that continuous weak solutions to \eqref{eqp} are also viscosity
	solutions.

	\begin{lemma}\label{equisols}
		Let $p>2,$ $g\in C(\partial\Omega)$   and $u$ be a continuous weak solution of \eqref{eqp}.
		Then $u$ is a viscosity solution of
		\[
			 \begin{cases}
			 	F_p(x, \nabla u, D^2 u)  =  0 & \text{in }  \quad  \{u > 0\} \cup \{u < 0\}, \\
   			\qquad H_p(x, \nabla u)  =  0 & \text{on }  \quad  \partial \Omega.
			\end{cases}
  		\]
	\end{lemma}

	\begin{proof}
		Let us proceed for the case of super-solutions.
		Fix $x_0 \in \overline{\Omega}$. We will divide the analysis into
		two cases:
		
		{\bf 1)} If $x_0 \in \Omega \cap ( \{u > 0\} \cup \{u < 0\})$.
			   In this case, let $\phi \in C^2(\Omega)$ be a test
				function such that $u(x_0) = \phi(x_0)$ and $u-\phi$ has
				a strict minimum at $x_0$. Our goal is to show that:
				$$
					 F_p(x_0, \nabla \phi(x_0), D^2 \phi(x_0)) \geq 0.
				$$
				Assume, for sake of contradiction that such a conclusion
				does not hold. Then, by continuity should exist a radius
				$\varrho>0$ such that
				$$
					F_p(x, \nabla \phi(x), D^2 \phi(x)) < 0 \quad
					\text{for all} \quad x \in
					B_{\varrho}= B_{\varrho}(x_0).
				$$
				Taking $\varrho$ smaller if necessary we can assume that
                $B_{\varrho}\subset\{u>0\}$ when $u(x_0)>0$ and $B_\rho\subset\{u<0\}$
                if $u(x_0)<0$.
	
				Now, consider $\displaystyle \iota \defeq \inf_{\partial
					B_{\varrho}} (u-\phi)(x)$ and $\Phi(x) \defeq
				\phi(x) + \frac{\iota}{10}$. Notice that such a function
				fulfils
				$$
					-\div(|\nabla \Phi|^{p-2}\nabla \Phi)<0 \qquad
					(pointwisely)\quad \text{in} \quad
					B_{\varrho}\qquad \text{and}\qquad u(x_0)< \Phi(x_0).
				$$
				Multiplying the previous inequality by $(\Phi-u)_{+}$
				(extended by zero outside $B_{\varrho}$) we obtain:
				\begin{equation}\label{Eq1}
					  \displaystyle \int_{\{\Phi>u\}\cap B_{\varrho}}
					  |\nabla \Phi|^{p-2}\nabla \Phi\cdot \nabla
					  (\Phi-u)dx<0.
				\end{equation}
				On the other hand, by taking $(\Phi-u)_{+}$ as test
				function in the weak formulation of \eqref{eqp} we obtain
				\begin{equation}\label{Eq2}
				  \displaystyle \int_{\{\Phi>u\}\cap B_{\varrho}}
				  |\nabla u|^{p-2}\nabla u\cdot \nabla (\Phi-u)dx=0.
				\end{equation}
				Next, subtracting \eqref{Eq1} from \eqref{Eq2} we get
				\begin{equation} \label{Eq3}
					\displaystyle  \int_{\{\Phi>u\}\cap B_\rho}
					\left(|\nabla \Phi|^{p-2}\nabla \Phi - |\nabla
					u|^{p-2}\nabla u\right) \cdot \nabla (\Phi-u) dx < 0.
				\end{equation}
				Finally, since the left hand side in \eqref{Eq3} is
				bounded by below by
				$$
					  \displaystyle C(N, p)\int_{\{\Phi>u\}\cap
					  	B_{\varrho}} |\nabla \Phi- \nabla u|^pdx \geq
					  	0,
				$$
				 this obligates $\Phi \leq u$ in $B_{\varrho}$. Such a
				 contradiction proves the desired result.

			  {\bf 2)} If $x_0 \in \partial \Omega$.
				Our goal now will be to show that:
				$$
					 \max\left\{F_p(x_0, \nabla \phi(x_0), D^2
					 \phi(x_0)), H_p(x_0, \nabla
					 \phi(x_0))\right\} \geq 0.
				$$
				Once again let us assume that such a conclusion is not
				true. Then, proceeding as before, we conclude that
				\begin{equation}\label{Eq4}
					  \displaystyle \int_{\{\Phi>u\}\cap
					  	B_{\varrho}} |\nabla
					  \Phi|^{p-2}\nabla \Phi\cdot \nabla (\Phi-u)dx
					  < \int_{\partial (\{\Phi>u\}\cap
					  	B_{\varrho}) \cap \partial \Omega}
					  g(\Phi-u)d\H,
				\end{equation}
				and
				\begin{equation}\label{Eq5}
					  \displaystyle \int_{\{\Phi>u\}} |\nabla
					  u|^{p-2}\nabla u \cdot\nabla (\Phi-u)dx \geq
					  \int_{\partial (\{\Phi>u\}\cap
					  	B_{\varrho}) \cap \partial \Omega}
					  g(\Phi-u)d\H.
				\end{equation}
				Therefore,
				$$
					\displaystyle C(N, p)\int_{\{\Phi>u\}\cap
						B_{\varrho}} |\nabla \Phi-
					\nabla u|^pdx  \leq   \displaystyle
					\int_{\{\Phi>u\}\cap
						B_{\varrho}} \left(|\nabla \Phi|^{p-2}\nabla
					\Phi - |\nabla u|^{p-2}\nabla u\right) \cdot \nabla
					(\Phi-u) dx
				    <  0,
				$$
				which again yields a contradiction. This proves that $u$
				is a viscosity supersolution.

		Similarly, one can prove that a continuous weak subsolution is a
		viscosity subsolution.
	\end{proof}

\section{The asymptotic analysis as $p \to \infty$.}\label{Sec3}

	Our first goal in this section is to obtain some (uniform in $p$)
	estimates on sequence of solutions to \eqref{eqp}. Taking into
	account that we are interested in the asymptotic behaviour as $p\to
	\infty$, we may assume that $p > N$ and, for this reason $u_p \in
	C^{0, 1-\frac{N}{p}}(\overline{\Omega})$ according
	to Sobolev embedding theorem.

    \begin{lemma}\label{LemmaWUniEst}
       Let $g\in C(\partial \Omega)$ be such that
       \[
            \int_{\partial \Omega} g(x)d\H >0,
       \]
        and
		$({u_p})_{p >N}$ be a sequence such that $u_p$ is a minimizers of $
		\mathcal{J}_p[\cdot]$ over $\mathbb{K}_{\alpha}^p.$ Then, up
        to a subsequence,
        $$
             u_p \to u_{\infty} \quad  \mbox{as} \quad p \to \infty,
        $$
        uniformly in $\overline{\Omega}$ and weakly in
        $W^{1, q}(\Omega)$ for all $q>1.$

        Furthermore, any possible limit $u_{\infty}$ is Lipschitz continuous
        with
        $$
            \|\nabla u_{\infty}\|_{L^{\infty}(\Omega)} \leq 1.
        $$
    \end{lemma}

    \begin{proof}
        By multiplying the equation by $u_p$ and integrating we obtain via H\"{o}lder
        inequality the following
        \begin{equation}\label{EqEst}
            \displaystyle \int_{\Omega} |\nabla u_p|^pdx =
            \int_{\partial \Omega} gu_{p}d\H
            \leq \|g\|_{L^{p^{\prime}}(\partial \Omega)}
            \|u_p\|_{L^{p}(\partial \Omega)}.
        \end{equation}
        Now, let us recall the trace inequality from \cite[Theorem 1, page 258]{Evans}
        $$
            \displaystyle \|u_p\|_{L^{p}(\partial \Omega)} \leq
            \sqrt[p]{pC_0}\|u_p\|_{W^{1, p}(\Omega)},
        $$
        where $C_0$ is a constant that does not depend on $p$. By substituting such estimate
        in \eqref{EqEst} we obtain
        \begin{equation}\label{EqEst2}
            \displaystyle \int_{\Omega} |\nabla u_p|^pdx \leq
            \sqrt[p]{pC_0}\|g\|_{L^{p^{\prime}}(\partial \Omega)}\|u_p\|_{W^{1, p}(\Omega)}.
        \end{equation}

        On the other hand,
        since $\Leb(\{u_p>0\})= \alpha< \Leb(\Omega)$ (see Theorem \ref{ThmExistMin}),
        for $p>N$ we get from Theorem \ref{MorIneq}
        the following
        \begin{equation}\label{Eq3.3}
            \|u_p\|_{L^p(\Omega)}
            \leq C(N, p, \Omega)\|\nabla u_p\|_{L^{p}(\Omega)},
        \end{equation}
        where $C(N, p, \Omega)$ is uniformly bounded in $p.$

       Connecting the estimate \eqref{Eq3.3} with \eqref{EqEst2} we conclude that
        $$
            \displaystyle \int_{\Omega} |\nabla u_p|^pdx \leq \sqrt[p]{pC_0}
            C(n, p,\Omega,)
            \|g\|_{L^{p^{\prime}}(\partial \Omega)}\|\nabla u_p\|_{L^{p}(\Omega)},
        $$
        which implies that
        $$
            \displaystyle \|\nabla u_p\|_{L^{p}(\Omega)} \leq
             \mathfrak{C}_p\left(\int_{\partial \Omega} |
             g|^{p^{\prime}}\right)^{\frac{1}{p}},
        $$
        where $\mathfrak{C}_p \to 1$ as $p \to \infty$.
        Now, fix $q>N$, and take $p > q$. Thus, we have
        \begin{equation}\label{eq:auxlq}
	         \displaystyle \|\nabla u_p\|_{L^{q}(\Omega)} \leq
            \Leb(\Omega)^{\frac{1}{q}-\frac{1}{p}}\|\nabla u_p\|_{L^{p}(\Omega)}
            \leq \mathfrak{C}_p \Leb(\Omega)^{\frac{1}{q}-\frac{1}{p}}
            \left(\int_{\partial \Omega} |g|^{p^{\prime}}\right)^{\frac{1}{p}}.
        \end{equation}
        Since $\mathfrak{C}_p \Leb(\Omega)^{\frac{1}{q}-\frac{1}{p}}
        \to \Leb(\Omega)^{\frac{1}{q}}$ as $p \to \infty,$
        we get that, up to a subsequence,
        $$
             u_p \to u_{\infty} \quad  \mbox{as} \quad p \to \infty,
        $$
        uniformly in $\overline{\Omega}$ and weakly in $W^{1, q}(\Omega).$
        Notice that, by \eqref{eq:auxlq},
        $$
             \|\nabla u_{\infty}\|_{L^{q}(\Omega)} \leq \Leb(\Omega)^{\frac{1}{q}}.
        $$
        Since that the previous inequality holds for every $q>N$, we conclude that
        $u_{\infty} \in W^{1, \infty}(\Omega)$.
        Furthermore, taking the limit as $q \to \infty$ we get
        $\|\nabla u_{\infty}\|_{L^{\infty}(\Omega)} \leq 1$.
    \end{proof}

 As a consequence, we obtain the following corollary.

    \begin{corollary} If $g \geq 0$, then $\| \nabla u_{\infty}\|_{L^{\infty}(\Omega)}= 1$.
    \end{corollary}

    \begin{proof} One more time by multiplying the equation by $u_p$, integrating, and using Lemma \ref{LemmaWUniEst} we obtain
        \begin{equation}\label{EqEst2}
            \displaystyle \lim_{p \to +\infty}\int_{\Omega} |\nabla u_p|^pdx =
            \lim_{p \to +\infty} \int_{\partial \Omega} gu_{p}d\H
            = \int_{\partial \Omega} gu_{\infty}d\H.
        \end{equation}
    Now, if we multiply the equation by a test function $\Theta$, we have by using the H\"{o}lder inequality (for $p \gg 1$ large enough) the following (for $\varepsilon(p) = \text{o}(1)$ as $p \to \infty$)
    $$
\begin{array}{rcl}
  \displaystyle \int_{\partial \Omega} g\Theta d\H & \leq  & \displaystyle \left(\int_{\Omega} |\nabla \Theta|^pdx\right)^{\frac{1}{p}}\left(\int_{\Omega} |\nabla u_p|^pdx\right)^{\frac{p-1}{p}} \\
   & \leq  & \displaystyle \left(\int_{\Omega} |\nabla \Theta|^pdx\right)^{\frac{1}{p}}\left(\int_{\partial \Omega} gu_{\infty}d\H + \varepsilon(p)\right)^{\frac{p-1}{p}}.
\end{array}
$$
Passing to the limit as $p \to \infty$ we conclude that
$$
  \displaystyle \int_{\partial \Omega} g\Theta d\H  \leq \|\nabla \Theta\|_{L^{\infty}(\Omega)}.\int_{\partial \Omega} gu_{\infty}d\H.
$$
Finally, by taking as test function $u_{\infty}$ itself and using once again Lemma \ref{LemmaWUniEst} we obtain as a consequence the desired conclusion.
\end{proof}

Now, we supply the proof of Theorem \ref{teo.1.intro}.

    \begin{proof}[{\bf Proof of Theorem \ref{teo.1.intro}}] By Lemma \ref{LemmaWUniEst}, up
        to a subsequence,
        $$
             u_p \to u_{\infty} \quad  \mbox{as} \quad p \to \infty,
        $$
        uniformly in $\overline{\Omega}$ and weakly in $W^{1, q}(\Omega)$ for all $q>1.$

        On the other hand, using a test function $\Theta$ with
        $\|\nabla \Theta\|_{L^p(\Omega)}\leq 1$, in the variational minimization
        problem solved by $u_p$ we obtain
        $$
          \displaystyle \frac1p \int_\Omega |\nabla \Theta|^p dx -
            \int_{\partial \Omega}g  \Theta d\H  \geq   \displaystyle  \frac1p \int_\Omega |\nabla u_p |^p dx - \int_{\partial \Omega} g u_p d\H  \geq  \displaystyle -  \int_{\partial \Omega}g u_p d\H.
        $$
        Passing to the limit as $p \to \infty$ we get that
        $$
            \displaystyle \int_{\partial \Omega} g \Theta d\H\leq
            \int_{\partial \Omega} gu_{\infty}d\H.
        $$
        Therefore, the limit function $u_{\infty}$ is a solution to the maximization problem
        $$
            \int_{\partial \Omega} gu_{\infty}d\H = \max\left\{ \int_{\partial \Omega}
            gv d\H\ \colon v \in W^{1, \infty}\left(\Omega\right),
            \|\nabla v\|_{L^{\infty}(\Omega)}\leq 1 \,\,\text{and} \,\,
            \Leb(\{v>0\}) \leq \alpha\right\}.
        $$
        This finishes the proof.
    \end{proof}

   \begin{remark}Notice that it is not immediate that a maximizer of
   $$
   \max\left\{ \int_{\partial \Omega}
            gv d\H\ \colon v \in W^{1, \infty}\left(\Omega\right),
            \|\nabla v\|_{L^{\infty}(\Omega)}\leq 1 \,\,\text{and} \,\,
            \Leb(\{v>0\}) \leq \alpha\right\}
   $$
   verifies
    $$\Leb(\{u_\infty>0\})=\alpha.$$
\end{remark}

    Now, we prove Theorem \ref{teo.2.intro}:

    \begin{proof}[{\bf Proof of Theorem \ref{teo.2.intro}}]
        First of all, let us verify that
        $$
           -\Delta_{\infty} u_{\infty} = 0 \quad \text{in} \quad \{u_{\infty}>0\} \cup
           \{u_{\infty}<0\}
       $$
        in the viscosity sense.

        We start proving that it is a subsolution. To this end,
        fix $x_0 \in \{u_{\infty}>0\} \cup \{u_{\infty}<0\}$ and
        let $\phi \in C^2(B_{\varepsilon}(x_0))$ (for $0<\varepsilon\ll1$)
        be a test function such that $u_{\infty}-\phi$ has a strict maximum at $x_0$.
        From uniform convergence, up to a subsequence, $u_{p} \to u_{\infty},$
        we get that for each $p\geq N$,  $u_{p}-\phi$ has a maximum at some point
        $x_{p} \in (\{u_{\infty}>0\} \cup \{u_{\infty}<0\})\cap B_{\varepsilon}(x_0)$,
        where $x_{p} \to x_0$. Since that $u_{p}$ is a weak subsolution
        (resp. viscosity subsolution according to Lemma \ref{equisols}) of
        $$
            -\Delta_{p} u_{p}  = 0 \quad \text{in} \quad \{u_{p}>0\} \cup \{u_{p}<0\}
        $$
        we get that
        $$
            F_{p}\left(x_{p}, \nabla \phi(x_{p}), D^2 \phi(x_{p})\right) \leq 0.
        $$

        Now, if $|\nabla \phi(x_0)| = 0$ then trivially we get
        $-\Delta_{\infty}\phi(x_0) \leq 0$. On the other hand, if
        $|\nabla \phi(x_0)| \neq 0$, then we have that $|\nabla \phi(x_{p})| \neq 0$
        for large values of $p$.
        Consequently
        $$
            -\nabla \phi(x_{p})^TD^2 \phi(x_{p})\cdot \nabla \phi(x_{p})
            \leq \frac{1}{p-2}|\nabla \phi(x_{p})|^{2}
            \Delta \phi(x_{p}).
        $$
        Finally, taking the limit as $p \to \infty$ in the above inequality
        we conclude that
        $$
            -\Delta_{\infty}\phi(x_0) \leq 0,
        $$
        showing $u_{\infty}$ is a viscosity subsolution, as desired.

        Similarly one can prove that $u_{\infty}$ is a viscosity supersolution.
        We omit this part here.

        \medskip

        Next, let us verify the limit profile at free boundary points. We will need the
        lower and upper semi-continuous envelopes, since the limit operator
        is discontinuous across the phase transitions.

        Fixed $x_0 \in \partial \{u_{\infty} = 0\} \cap \Omega$,
        let $\phi \in C^2(B_{\varepsilon}(x_0))$ be such that
        $u_{\infty}(x_0) = \phi(x_0)=0$ and $u_{\infty}(x) < \phi(x)$ holds for
        $x \neq x_0$ in $B_{\varepsilon}(x_0)$.
        We would like to prove the following
        $$
            F_{\ast}(x_0, \nabla \phi(x_0), D^2 \phi(x_0)) \leq 0,
        $$
        where
        $$
            F_{\ast}(x_0, \nabla \phi(x_0), D^2 \phi(x_0))
            \defeq \min\{\phi(x_0), -\Delta_{\infty}\phi(x_0)\}
        $$
        is the \textit{lower semi-continuous envelope} of $F_{\infty}$ in
        $B_{\varepsilon}(x_0)$. As before, there exists a sequence
        $B_{\varepsilon}(x_0) \ni x_{p} \to x_0$ such that $u_{p}-\phi$ has a local
        maximum at $x_{p}$. If $\nabla \phi(x_0) = 0$, then there is nothing to proof.
        Now, if $|\nabla \phi(x_0)| \neq 0$ we must consider two possibilities:

       {\bf Case 1.} If $u_{p_j}(x_{p_j})<0$ or $u_{p_j}(x_{p_j}) > 0$ for a subsequence
                $(p_j)_{j\geq 1}$. In this case, since $u_{p_j}$ is a weak sub-solution
                (resp. viscosity super-solution) to \eqref{eqp}, we have that
                $$
                        F_{p_j}\left(x_{p_j}, \nabla \phi(x_{p_j}), D^2 \phi(x_{p_j})\right) \leq 0.
                $$
                Finally, passing to the limit as $p_j \to \infty$ we obtain
                $$
                    -\Delta_{\infty} \phi(x_0) \leq 0.
                $$

            {\bf Case 2.} If $u_{p_j}(x_{p_j})  =  0$ for a subsequence $(p_j)_{j\geq 1}$.
                In this case the conclusion is immediate since using continuity we get $\phi(x_0) = 0$.

        For the super-solution case fix $x_0 \in \partial \{u_{\infty} = 0\} \cap \Omega$
        and $\phi \in C^2(B_{\varepsilon}(x_0))$ such that $u_{\infty}(x_0) = \phi(x_0)=0$
        and $u_{\infty}(x) > \phi(x)$ holds for $x \neq x_0$ in $B_{\varepsilon}(x_0)$.
        This time we would like to prove the following:
        $$
            F^{\ast}(x, \nabla \phi(x_0), D^2 \phi(x_0)) \geq 0,
        $$
        where
        $$
            F^{\ast}(x, \nabla \phi(x_0), D^2 \phi(x_0))
            \defeq \max\{\phi(x_0), -\Delta_{\infty}\phi(x_0)\}
        $$
        is the \textit{upper semi-continuous envelope} of $F_{\infty}$ in $\Omega$.
        The analysis for this case runs similarly to previous one.

    Next, we deal with the boundary condition.
    First, let $\phi \in C^2(\overline{\Omega})$ be a test function and assume that $u_{\infty}-\phi$ has a strict minimum at $x_0 \in \partial \Omega$ with $u_{\infty}(x_0)=\phi(x_0)\neq 0$ and $g(x_0)>0$. One more time, from uniform convergence $u_{p_j} \to u_{\infty}$ we obtain that $u_{p_j}-\phi$ has a minimum at some point $x_{p_j} \in \overline{\Omega} $, where $x_{p_j} \to x_0$. Now, if $x_{p_j} \in \Omega$ for infinitely many values of $j$, then by arguing as before we conclude that
    $$
    -\Delta_{\infty}\phi(x_0) \geq 0 \quad (\text{resp.} \quad \max\{-\Delta_{\infty}\phi(x_0), \phi(x_0)\} \geq 0 \quad \text{at free boundary points}).
    $$
    However, if $x_{p_j} \in \partial \Omega$, then we have, from Remark \ref{RemMonotB_p}, that
    $$
    H_{p_j}(x_{p_j}, \nabla \phi(x_{p_j})) \geq 0.
    $$
    Taking into account that $g(x_0)>0$, then $\nabla \phi(x_0) \neq 0$, and we obtain
    $$
    |\nabla \phi(x_0)| \geq 1 \quad \text{and} \quad \nabla \phi(x_0)\cdot \eta(x_0) \geq 0.
    $$
    In conclusion, if $u_{\infty}-\phi$ has a strict minimum at $x_0 \in \partial \Omega$ with $g(x_0)>0$, then we have the following inequality
    $$
    \max\left\{  -\Delta_{\infty}\phi(x_0), \,\,\,\min\left\{|\nabla \phi(x_0)| - 1, \,\,\,\frac{\partial \phi}{\partial \eta}(x_0) \right\}\,\,\right\} \geq 0,
    $$
    $$
    \left(\text{resp.}\,\,\,\max\left\{  \max\{-\Delta_{\infty}\phi(x_0), \phi(x_0)\}, \,\,\,\min\left\{|\nabla \phi(x_0)| - 1, \,\,\,\frac{\partial \phi}{\partial \eta}(x_0) \right\}\,\,\right\} \geq 0\,\,\, \text{at free boundary points}\right).
    $$

    For the next case, let us assume that $u_{\infty}-\phi$ has a strict maximum at $x_0 \in \partial \Omega$
    with $u_{\infty}(x_0)=\phi(x_0)\neq 0$ and $g(x_0)>0$. With the same notations as before, if $x_{p_j} \in \Omega$ for infinitely many $j$, then we conclude that
    $$
    -\Delta_{\infty}\phi(x_0) \leq 0 \quad (\text{resp.} \quad \min\{-\Delta_{\infty}\phi(x_0), \phi(x_0)\} \leq 0 \quad \text{at free boundary points}).
    $$
    On the other hand, when $x_{p_j} \in \partial \Omega$, using
    $$
    H_{p_j}(x_{p_j}, \nabla \phi(x_{p_j})) \leq 0,
    $$
    we get that, if $\nabla \phi(x_0)-1>0$, then $\frac{\partial \phi}{\partial \eta}(x_0) \geq 0$. We have that the following inequality holds
    $$
    \min\left\{-\Delta_{\infty}\phi(x_0), \,\,\,\min\left\{|\nabla \phi(x_0)| - 1, \,\,\, \frac{\partial \phi}{\partial \eta}(x_0) \right\}\,\,\right\} \leq 0,
    $$
    $$
    \left(\text{resp.}\,\,\,\min\left\{\min\{-\Delta_{\infty}\phi(x_0), \phi(x_0)\}, \,\,\,\min\left\{|\nabla \phi(x_0)| - 1, \,\,\, \frac{\partial \phi}{\partial \eta}(x_0) \right\}\,\,\right\} \leq 0 \,\,\, \text{at free boundary points}\right).
    $$

    The case in which $u_{\infty} - \phi$ has a strict maximum / minimum at $x_0 \in \{g<0\}$
    with $u_{\infty}(x_0)=\phi(x_0)\neq 0$ can be handled similarly.

    Now, if $u_{\infty}-\phi$ has a strict minimum at $x_0 \in \partial \Omega$ with $u_{\infty}(x_0)=\phi(x_0)\neq 0$ and $x_0 \in \{g=0\}^{\circ}$ then we have
    $$
    H_{p_j}(x_{p_j}, \nabla \phi(x_{p_j}))\geq 0.
    $$
    Thus, by passing to the limit we obtain $\frac{\partial \phi}{\partial \eta}(x_0)\geq 0$. Therefore, the following inequality holds
    $$
    \max\left\{  -\Delta_{\infty}\phi(x_0),  \,\,\,\frac{\partial \phi}{\partial \eta}(x_0) \right\} \geq 0
    $$
    $$
    \left(\text{resp.}\,\,\,\max\left\{  \max\{-\Delta_{\infty}\phi(x_0), \phi(x_0)\}, \,\,\,\frac{\partial \phi}{\partial \eta}(x_0) \right\} \geq 0\,\,\,\text{at free boundary points}\right).
    $$

    Now, if $u_{\infty}-\phi$ has a strict maximum at $x_0 \in \partial \Omega$ with $u_{\infty}(x_0)=\phi(x_0)\neq 0$ and $x_0 \in \{g=0\}^{\circ}$ then we have
    $$
    H_{p_j}(x_{p_j}, \nabla \phi(x_{p_j}))\leq 0.
    $$
    Thus, by taking the limit as $p_j \to \infty$ we obtain $\frac{\partial \phi}{\partial \eta}(x_0)\leq 0$. Therefore, the following inequality holds
    $$
    \min\left\{ -\Delta_{\infty}\phi(x_0), \,\,\,\frac{\partial \phi}{\partial \eta}(x_0) \right\} \leq 0,
    $$
    $$
    \left(\text{resp.}\,\,\,\min\left\{    \min\{-\Delta_{\infty}\phi(x_0), \phi(x_0)\}, \,\,\,\frac{\partial \phi}{\partial \eta}(x_0) \right\} \leq 0\,\,\,\text{at free boundary points}\right).
    $$

    Finally, we just observe that we can handle the cases in which $u_{\infty}(x_0)=\phi(x_0)\neq 0$ and $x_0
    \in \partial \{g>0\}$ with $g(x_0)=0$, $x_0
    \in \partial \{g<0\}$ with $g(x_0)=0$ or $x_0
    \in \partial \{g>0\}\cap \partial \{g<0\}$ with $g(x_0)=0$ considering that the involved sequence $x_{p_j}$
    can be such that $g(x_{p_j})>0$, $g(x_{p_j})<0$ or $g(x_{p_j})=0$. Notice that in these cases
    we find the upper (or lower) semicontinuous envelope of $H$ that involve that $\max$ or the $\min$ of the
    previous cases. We leave the details to the reader.
    \end{proof}

%%%%%%%%%%%%%%%%%%%%%%%%%%%%%%%%%%%%%%%%%%%%%%%%%%%%%%%%%%%%%%%%%%%%%%%%%%%%%%%%%%%%%%%%%%%%

\section{Proof of the Monge-Kantorovich type results}\label{Sec4}

   In this short section we include the proof of our Monge-Kantorovich type results.
    The datum $g$ is assumed to be nonnegative, and therefore the same property holds true for the solutions $u_p$ (see Remark \ref{rem.2.6}).

    \begin{proof}[{\bf Proof of Theorem \ref{teo.3.intro}}]
        Following \cite{EvGang} we define the transport set for a maximizer
        $u_{\infty}$ of \eqref{eqMax_infty}:
        $$
            \mathrm{T}(u_{\infty}) \coloneqq
            \left\{ x \in \Omega \colon \exists \,\,y \in \partial \Omega
            \mbox{ with } |u_{\infty}(x) - u_{\infty}(y) |= |x-y| \right\}.
        $$
        Moreover, we define a transport ray by
        $$
            \mathrm{R}_{x} \coloneqq
            \{w \in \mathrm{T}(u_{\infty})
            \colon |u_{\infty}(x) - u_{\infty}(w)| = |x - w|\}.
        $$
        Observe that any two transport rays cannot intersect in $\Omega$, unless
        they are identical. In fact, assume $w \in \mathrm{T}(u_{\infty})$,
        and that there exist $x, y \in \Omega$ such that
        $$
            u_{\infty}(x) - u_{\infty}(w) = |x-w| \quad
             \text{and} \quad  u_{\infty}(w) - u_{\infty}(y) = |w-y|.
        $$
        Hence, from Lipschitz continuity for $u_{\infty}$ we obtain
        $$
                |x - y| \leq |x-w| + |w-y|
                = u_{\infty}(x) - u_{\infty}(y)
                \leq |u_{\infty}(x) - u_{\infty}(y)|
                \leq |x-y|,
        $$
        which is impossible, unless that $x, y$ and $w$ are collinear points.

        Now, we observe that for each $u_p$ there exists a sequence $\epsilon_j \to 0+$ as $j \to +\infty$ such that the set $\mathcal{S}_j \defeq \{ u_p > \epsilon_j \}$ has finite perimeter for every $j \in \mathbb{N}$ (cf. \cite[Theorem 1, \S 5.5]{EvGar}). Hence, there is a measure supported on the set
        $$
            \partial \{ u_p > \epsilon_j\} \cap \Omega
        $$
        defined by
        $$
            \nu_{p,\epsilon_j} = |\nabla u_p|^{p-2} \frac{\partial u_p}{\partial \eta},
        $$
        where $\eta$ is the unit outer normal to $\partial \{ u_p >\epsilon_j \} \cap \Omega$.
        Moreover, this measure is non-negative and verifies
        $$
            \int_{\Omega} d\nu_{p,\epsilon_j}  = \int_{\partial \Omega \cap \{ u_p >\epsilon_j \}} gd\H.
        $$
        In fact, to show this identity one just have to recall that
        $\Delta_p u_p=0$ in $\{ u_p > \epsilon_j\}$.

        Now to obtain the measure $\nu_\infty$ we just have to take the limit
        (along a subsequence if necessary)
        of $\nu_{p,\epsilon_j}$ (first we take $\epsilon_j \to 0+$ and then $p\to \infty$). This limit measure $\nu_\infty$ is supported on
        $$
            \partial \{ u_\infty >0\} \cap \Omega
        $$
        and verifies the compatibility condition
        $$
            \int_{\partial \{ u_\infty >0\} \cap \Omega} d\nu_\infty = \int_{\partial \Omega} gd\H.
        $$
        As the transport rays do not intersects, using our previous results, we obtain that
        $$
           \int_{\partial \Omega} u_\infty gd\H= \int_{\overline{\Omega}} u_\infty (gd\H - d\nu_\infty) = \max_{\omega}
            \left\{ \int_{\overline{\Omega}} \omega (gd\H - d\nu_\infty)  \right\}.
        $$
        where the maximum is taken in the set of $1-$Lipschitz functions:
        $$
            \displaystyle 1-\text{Lip}(\overline{\Omega})
            \coloneqq \left\{\Phi\colon \overline{\Omega} \to \R \colon
            \sup_{x, y \in \overline{\Omega}, \, \atop{x\neq y}} \frac{|\Phi(x)-\Phi(y)|}{|x-y|}
            \leq 1  \right\}.
        $$
        Finally, we notice that, since $\Leb(\{v_\infty >0\}) \leq \alpha$, we get that the
        transport set associated to this optimal transport problem has the property
        $\Leb (\mathrm{T}(u_{\infty}))\leq \alpha$.
\end{proof}

    Finally, we supply the proof of Theorem \ref{teo.5.intro}.

\begin{proof}[{\bf Proof of Theorem \ref{teo.5.intro}}]

        Now, our aim is to compute the maximum among every possible transport costs
        of $\mu = g\H \llcorner \partial \Omega$ to $\nu$ with the restriction that the transport set has measure less or equal than $\alpha$, that is,
$$
   \displaystyle \mathrm{W}^1_{\alpha}(\mu, \nu) \defeq \max_{\nu \in \mathcal{M}(\Omega), \,\omega \in 1-\text{Lip}(\overline{\Omega}), \atop{\Leb (\mathrm{T}(\omega))\leq \alpha}} \left\{\int_{\overline{\Omega}} \omega d(\mu-\nu) \right\}.
$$
    To this end, we just notice that $\nu_\infty$ (our limit measure) is a competitor in this maximization problem and hence the total cost for the limit problem verifies
$$
 \displaystyle  \int_{\partial \Omega} u_\infty gd\H = \int_{\overline{\Omega}} u_\infty (gd\H - d\nu_\infty) \leq \max_{\nu \in \mathcal{M}(\Omega), \, \omega \in 1-\text{Lip}(\overline{\Omega}), \atop{ \Leb (\mathrm{T}(\omega))\leq \alpha}} \left\{\int_{\overline{\Omega}} \omega d(\mu-\nu) \right\}.
$$
    Now, notice that, since we have that the total mass of $\nu$ is equal to $\displaystyle \int_{\partial \Omega} gd\H$, we can add a constant to $\omega$ (if necessary) and assume that $\displaystyle \inf_{\mathrm{T(\omega)}} \omega =0$.  Hence,
$$
\begin{array}{rcl}
\displaystyle \max_{\nu \in \mathcal{M}(\Omega),\,\omega \in 1-\text{Lip}(\overline{\Omega}), \atop{ \Leb (\mathrm{T}(\omega))\leq \alpha}} \left\{\int_{\overline{\Omega}} \omega d(\mu-\nu) \right\} & = & \displaystyle \max_{\omega \in 1-\text{Lip}(\overline{\Omega}), \atop{ \Leb (\mathrm{T}(\omega))\leq \alpha}} \max_{\nu \in \mathcal{M}(\Omega)}\left\{\int_{\overline{\Omega}} \omega d(\mu-\nu) \right\} \\
 & \leq & \displaystyle \max_{\omega \in 1-\text{Lip}(\overline{\Omega}), \atop{ \Leb (\mathrm{T}(\omega)) \leq \alpha}}\left\{\int_{\partial \Omega} \omega g d\H\right\}\\
 & = & \displaystyle \int_{\partial \Omega} u_\infty g d\H.
\end{array}
$$
    Therefore, we conclude that the obtained limit cost (the total cost of the transport of $g\H \llcorner \partial \Omega$ to $\nu_\infty$) gives the maximum possible among transport costs to nonnegative measures $\nu$ with measure of the involved transport set less or equal than $\alpha$.
    \end{proof}

%%%%%%%%%%%%%%%%%%%%%%%%%%%%%%%%%%%%%%%%%%%%%%%%%%%%%%%%%%%%%%%%%%%%%%%%%%%%%%%%%%%%%%%%%%%%

\section{Examples}\label{Sec5}

    \begin{example}
        Consider the domain $\Omega = (-1, 1)$ and
        the boundary datum such that  $g(1)=g(-1)=A>0$.
        Thus, for fixed $\alpha \in (0, 2)$ and $t \in (0, 1)$
        the weak solution of
        $$
            \left\{
                \begin{array}{ll}
                    -(|u_p^{\prime}(x)|^{p-2}u_p^{\prime}(x))^{\prime}  =  0
                    & \text{in }  \quad (-1, \,t\alpha-1) \cup (1-(1-t)\alpha, \, 1),\\
                    u_p= 0 & \text{in } \quad \left[t\alpha-1, \,1-(1-t)\alpha\right],\\
                    |u_p^{\prime}(\pm 1)|^{p-2}u_p^{\prime}(\pm 1) \eta(\pm 1)  =  A,
                \end{array}
            \right.
        $$
        (notice that $u_p$ satisfies the volume constraint $\Leb(\{u_p>0\}) = \alpha$) is given by
        $$
            u_p(x) =
            \left\{
                \begin{array}{ll}
                     A^{\frac{1}{p-1}}\left[(t\alpha-1)-x\right] &
                     \text{if }  \quad x \in (-1,\, t\alpha-1),\\
                        0 & \text{if } \quad x\in \left[t\alpha-1, \,1-(1-t)\alpha\right],\\
                    A^{\frac{1}{p-1}}\left\{x-\left[1-(1-t)\alpha\right]\right\}
                    & \text{if }  \quad x \in (1-(1-t)\alpha, 1).
                \end{array}
            \right.
        $$
        Letting $p \to \infty$, we obtain the limiting profiles, for $t \in (0, 1)$,
        $$
            u_{\infty}(x) =
                \left\{
                    \begin{array}{ll}
                      (t\alpha-1)-x & \text{if } \quad  x \in (-1, \,t\alpha-1)\\
                      0 & \text{if } \quad x\in \left[t\alpha-1, \,1-(1-t)\alpha\right]\\
                        x-\left[1-(1-t)\alpha\right] & \text{if } \quad
                        x \in (1-(1-t)\alpha, \,1).
                    \end{array}
                \right.
         $$
         Notice that in this example we do not have uniqueness of a limit profile.
         Also note that the limit profiles are independent of $A$.

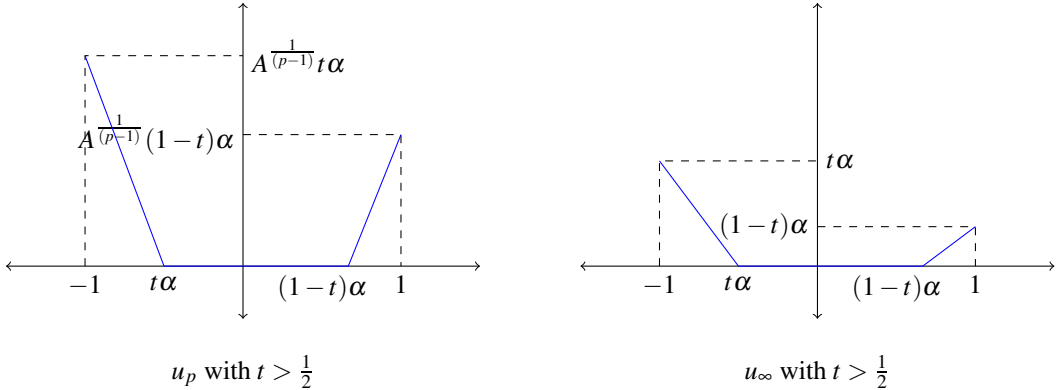
         \begin{multicols}{2}
	        \begin{tikzpicture} [scale=.70]
		        \draw[arrows=<->](0,-1)--(0,5);
   		        \draw[arrows=<->](-4.5,0)--(4.5,0);
 		        \draw[dashed] (-3,0) node[below]{$-1$} -- (-3,4) -- (0,4)
 		        node[right]{$A^{\frac{1}{(p-1)}}t\alpha$};
 		         \draw[dashed] (0,2.5)
 		         node[left]{$A^{\frac{1}{(p-1)}}(1-t)\alpha$} -- (3,2.5) -- (3,0)
 		        node[below]{$1$}  ;
 		        \draw [blue] (-3,4) -- (-1.5,0) --  (2,0) -- (3,2.5);
 		        \draw (-1.5,0)node[below]{$t\alpha$};
 		         \draw (1.5,0)node[below]{$(1-t)\alpha$};
 		        \draw (0,-2) node {$u_p$ with $t>\frac12$};
            \end{tikzpicture}

         \begin{tikzpicture} [scale=.70]
		        \draw[arrows=<->](0,-1)--(0,5);
   		        \draw[arrows=<->](-4.5,0)--(4.5,0);
 		        \draw[dashed] (-3,0) node[below]{$-1$} -- (-3,2) -- (0,2)
 		        node[right]{$t\alpha$};
 		         \draw[dashed] (0,.75)
 		         node[left]{$(1-t)\alpha$} -- (3,.75) -- (3,0)
 		        node[below]{$1$}  ;
 		        \draw[blue] (-3,2)--(-1.5,0)--(2,0)--(3,.75);
 		        \draw  (-1.5,0)node[below]{$t\alpha$};
 		        \draw (1.5,0)node[below]{$(1-t)\alpha$};
 		        \draw (0,-2) node {$u_\infty$ with $t>\frac12$};
            \end{tikzpicture}
        \end{multicols}
    \end{example}

    \begin{example} \label{rem.33}
        We could also consider in the previous example the case in which
        $g(-1)> g(1) > 0$. In this case, we obtain a unique minimizer
        $$
            u_p(x) = g(-1)^{\frac{1}{p-1}}[(\alpha-1)-x]_{+}
        $$
        and
        $$
            u_{\infty}(x) = [(\alpha-1)-x]_{+}
        $$
        as the unique limit as $p\to \infty$ (remark that this functions is also the unique solution to our limiting optimization problem). Note that in this case we have uniqueness
        of the limit profiles.

        Also notice that in this case the boundary condition $|u_p^{\prime}(x)|^{p-2}u_p^{\prime}(x) = g(x)$ holds only at $x=-1$ since at $x=1$ we have $u_p(1)=0$ and $|u_p^{\prime}(1)|^{p-2} u_p^{\prime}(1) = 0  \neq g(1)$.
\end{example}

\subsection*{Acknowledgments}
J.V. da Silva thanks Dept. Math/FCEyN from Universidad de Buenos Aires for providing a nice working environment during his Postdoctoral program. This work has been partially supported by Consejo Nacional de Investigaciones Cient\'{i}ficas y T\'{e}cnicas (CONICET-Argentina).


\begin{thebibliography}{99}

\bibitem{AAC} {\sc{ N. Aguilera, H. Alt and L. Caffarelli}}. \textit{An optimization problem with volume constraint}. SIAM J. Control Optim. 24 (1986), no. 2, 191-198.\label{AAC}

\bibitem{Amb} {\sc{L. Ambrosio}}. \textit{Lecture notes on optimal transport problems}. Mathematical aspects of evolving interfaces (Funchal, 2000), 1-52, Lecture Notes in Math., 1812, Springer, Berlin, 2003.

\bibitem{ACBBV} {\sc{L. Ambrosio,; L.A. Caffarelli; Y. Brenier; G. Buttazzo; C. Villani}}. \textit{Optimal transportation and applications}. Lectures from the C.I.M.E. Summer School held in Martina Franca, September 2-8, 2001. Edited by Caffarelli and S. Salsa. Lecture Notes in Mathematics, 1813. Springer-Verlag, Berlin; Centro Internazionale Matematico Estivo (C.I.M.E.), Florence, 2003. viii+164 pp. ISBN: 3-540-40192-X.\label{ACBBV}

\bibitem{Arons} {\sc{ G. Aronsson}}. \textit{Extension of functions satisfying Lipschitz conditions}. Ark. Mat. 6 1967 551-561 (1967).\label{Arons}

\bibitem{ACJ} {\sc{ G. Aronsson, M.G. Crandall and P. Juutinen}}. \textit{A tour of the theory of absolutely minimizing functions}. Bull. Amer. Math. Soc. (N.S.) 41 (2004), no. 4, 439-505.\label{ACJ}

\bibitem{Barl} {\sc{ G. Barles}}. \textit{Fully nonlinear Neumann type boundary conditions for second-order elliptic and parabolic equations}. J. Differential Equations 106 (1993), no. 1, 90-106.

\bibitem{BucBut} {\sc{ D. Bucur and G. Buttazzo}}. \textit{Variational Methods in Shape Optimization Problems}. Progress in Nonlinear Differential Equations and their Applications, 65. Birkh\"{a}user Boston, Inc., Boston, MA, 2005. viii+216 pp. ISBN: 978-0-8176-4359-1; 0-8176-4359-1.

\bibitem{CIL} {\sc{ M.G. Crandall, H. Ishii and P.L. Lions}}. \textit{User's guide to viscosity solutions of second order partial differential equations}. Bull. Amer. Math. Soc. (N.S.) 27 (1992), no. 1, 1-67.

\bibitem{Cran} {\sc{ T. R. Cranny}}. \textit{Regularity of solutions for the generalized inhomogeneous Neumann boundary value problem}. J. Differential Equations 126 (1996), no. 2, 292-302.

\bibitem{daSR} {\sc{  J.V. da Silva and J.D. Rossi}}. \textit{The limit as $p \to \infty$ in free boundary problems with fractional $p-$Laplacians}. To appear in Trans. Amer. Math. Soc. DOI: \url{https://doi.org/10.1090/tran/7559}.

\bibitem{Diaz} {\sc{ J.I. Diaz}}. \textit{Nonlinear partial differential equations and free boundaries. Vol I}. Elliptic equations. Research Notes in Mathematics, 106. Pitman (Advanced Publishing Program), Boston, MA, 1985. vii+323 pp. ISBN: 0-273-08572-7.

\bibitem{Ev99} {\sc{ L.C. Evans}}. \textit{Partial differential equations and Monge-Kantorovich mass transfer}. Current developments in mathematics, 1997 (Cambridge, MA), 65-126, Int. Press, Boston, MA, 1999.

\bibitem{Evans} {\sc{ L.C. Evans}}. \textit{Partial differential equations}. Graduate Studies in Mathematics, 19. American Mathematical Society, Providence, RI, 1998. xviii+662 pp. ISBN: 0-8218-0772-2.

\bibitem{EvGang} {\sc{ L.C. Evans and W. Gangbo}}. \textit{Differential equations methods for the Monge-Kantorovich mass transfer problem}. Mem. Amer. Math. Soc. 137 (1999), no. 653, viii+66 pp.

\bibitem{EvGar} {\sc{ L.C. Evans and R.F Gariepy}}. \textit{Measure theory and fine properties of functions}.
Studies in Advanced Mathematics. CRC Press, Boca Raton, FL, 1992. viii+268 pp. ISBN: 0-8493-7157-0.

\bibitem{BMW} {\sc{ J. Fern\'{a}ndez Bonder, S. Mart\'{i}nez and N. Wolanski}}. \textit{An optimization problem with volume constraint for a degenerate quasilinear operator}. J. Differential Equations 227 (2006), no. 1, 80-101.\label{BMW}

\bibitem{G-AMPR} {\sc{ J. Garc\'{i}a-Azorero, J.J. Manfredi, I. Peral and J.D. Rossi}}. \textit{The Neumann problem for the $\infty-$Laplacian and the Monge-Kantorovich mass transfer problem}. Nonlinear Anal. 66 (2007), no. 2, 349-366.

\bibitem{G-AMPR3} {\sc{ J. Garc\'{i}a-Azorero, J.J. Manfredi, I. Peral and J.D. Rossi}}. \textit{Limits for Monge-Kantorovich mass transport problems}. Commun. Pure Appl. Anal. 7 (2008), no. 4, 853-865.

\bibitem{IMRT} {\sc{ N. Igbida, J.M. Maz\'{o}n, J.D. Rossi and J. Toledo}}. \textit{A Monge-Kantorovich mass transport problem for a discrete distance}. J. Funct. Anal. 260 (2011), no. 12, 3494-3534.

\bibitem{Ish} {\sc{ H. Ishii}}. \textit{Fully nonlinear oblique derivative problems for nonlinear second-order elliptic PDEs}. Duke Math. J. 62 (1991), no. 3, 633-661.

\bibitem{IL} {\sc{ H. Ishii and P.L. Lions}}. \textit{Viscosity solutions of fully nonlinear second-order elliptic partial differential equations}. J. Differential Equations 83 (1990), no. 1, 26-78.

\bibitem{Kant}  {\sc{ L.V. Kantorovich}}. \textit{On the tranfer of masses}. Dokl. Nauk. SSSR 37 (1942), 227-229.

\bibitem{Kawohl} {\sc{ B. Kawohl}}. \textit{Rearrangements and Convexity of Level Sets in PDE}. Lecture Notes in Mathematics, 1150. Springer-Verlag, Berlin, 1985. iv+136 pp. ISBN: 3-540-15693-3.\label{Kawohl}

\bibitem{OliTei} {\sc{ K. Oliveira and E.V. Teixeira}}. \textit{An optimization problem with free boundary governed by a degenerate quasilinear operator}. Differential Integral Equations 19 (2006), no. 9, 1061-1080.

\bibitem{MRT3} {\sc{J. M. Maz\'{o}n, J. D. Rossi and J. Toledo}}. \textit{An optimal transportation problem with a cost given by the Euclidean distance plus import/export taxes on the boundary}. Rev. Mat. Iberoam. 30 (2014), no. 1, 277-308.

\bibitem{Prat} {\sc{ A. Pratelli}}. \textit{On the equality between Monge's infimum and Kantorovich's minimum in optimal mass transportation}. Ann. Inst. H. Poincar\'{e} Probab. Statist. 43 (2007), no. 1, 1-13.

\bibitem{RT} {\sc{ J.D. Rossi and E.V. Teixeira}}. \textit{A limiting free boundary problem ruled by Aronsson's equation}. Trans. Amer. Math. Soc. 364 (2012), no. 2, 703-719.

\bibitem{RW}  {\sc{ J.D. Rossi and P. Wang}}. \textit{The limit as $p \rightarrow \infty$ in a two-phase free boundary problem for the $p$-Laplacian}. Interfaces Free Bound. 18 (2016), no. 1, 115-135.\label{RW}

\bibitem{Sperner} {\sc{ E. Sperner Jr}}. \textit{Spherical Symmetrization and Eigenvalue Estimates}. Math. Z. 176 (1981), no. 1, 75-86.\label{Sperner}

\bibitem{Tei1} {\sc{ E.V. Teixeira}}. \textit{The nonlinear optimization problem in heat conduction}. Calc. Var. Partial Differential Equations 24 (2005), no. 1, 21-46.\label{Tei1}

\bibitem{Tei2} {\sc{ E.V. Teixeira}}. \textit{Uniqueness, symmetry and full regularity of free boundary in optimization problems with volume constraint}. Interfaces Free Bound. 9 (2007), no. 1, 133-148.\label{Tei2}

\bibitem{Tei3} {\sc{ E.V. Teixeira}}. \textit{Optimal design problems in rough inhomogeneous media. Existence theory}. Amer. J. Math. 132 (2010), no. 6, 1445-1492.\label{Tei3}

\bibitem{Vil03} {\sc{ C. Villani}}. \textit{Topics in Optimal Transportation}. Graduate Studies in Mathematics, 58. American Mathematical Society, Providence, RI, 2003. xvi+370 pp. ISBN: 0-8218-3312-X.

\bibitem{Vil09} {\sc{ C. Villani}}. \textit{Optimal transport, old and new}. Grundlehren der Mathematischen Wissenschaften [Fundamental Principles of Mathematical Sciences], 338. Springer-Verlag, Berlin, 2009. xxii+973 pp. ISBN: 978-3-540-71049-3.

\end{thebibliography}
\end{document}